\documentclass[9pt]{amsart}

\usepackage{bbm}
\usepackage{mathtools}
\usepackage{color}
\usepackage{graphicx}
\usepackage{verbatim}

\usepackage{tikz}
\usetikzlibrary{arrows}

\def\red{\color{red}}
\tikzstyle{state}=[shape=circle,fill=none,draw=black,minimum size=0.5cm,line width=1pt]
\tikzstyle{none}=[shape=circle,minimum size=0.1cm,fill=none, draw=none]
\tikzstyle{tran}=[line width=1pt, draw=black]
\graphicspath{{Fig/}}

\newtheorem{prop}{Proposition}
\newtheorem{theorem}{Theorem}
\newtheorem{lemma}{Lemma}

\newtheorem*{assumption}{Assumptions A${}^*$\hspace{-1.5mm}}

\newtheorem{corollary}{Corollary}

\def\N{{\mathbb N}}
\def\R{{\mathbb R}}

\def\P{{\mathbb P}}
\def\E{{\mathbb E}}

\def\eps{\varepsilon}

\newcommand{\diff}{\mathop{}\mathopen{}\mathrm{d}}
\newcommand\ind[1]{\mathbbm{1}_{\left\{#1\right\}}}
\newcommand\croc[1]{\left\langle #1\right\rangle}
\newcommand\steq[1]{\stackrel{\text{\rm #1.}}{=}}

\def\cal{\mathcal}
\def\eps{\varepsilon}

\title[Nucleation of Polymerization Processes]{On the Asymptotic Distribution of Nucleation Times of Polymerization Processes}

\address[Ph. Robert, W. Sun]{INRIA Paris, 2 rue Simone Iff, F-75012 Paris, France}
\author{Philippe  Robert}
\email{Philippe.Robert@inria.fr}
\urladdr{https://www-rocq.inria.fr/who/Philippe.Robert}
\author{Wen Sun${}^*$}\thanks{* The author's work has been partially supported by a public grant overseen by the French National Research Agency (ANR) as part of the ``Investissements d'Avenir'' program (reference: ANR-10-LABX-0098).}
\email{Wen.Sun@inria.fr}
\date{\today}

\begin{document}

\begin{abstract}
  In this paper, we investigate a stochastic model describing the time evolution of a polymerization process. A polymer is a macro-molecule resulting from the aggregation of several elementary sub-units called  monomers. Polymers can grow by addition of monomers or can be split into several polymers.  The initial state of the system consists mainly of monomers.  We study the time evolution  of the mass of polymers, in particular the asymptotic distribution of the first instant when the fraction of monomers used in polymers is above some positive threshold $\delta$. A scaling approach is used by taking the mass $N$ as a scaling parameter. The mathematical model used in this paper includes {\em a nucleation property}: If $n_c$ is defined as the size of the nucleus, polymers with a size less than $n_c$ are quickly fragmented into smaller polymers, at a rate proportional to $\Phi(N)$ for some  non-decreasing and unbounded function $\Phi$. For polymers of size greater than $n_c$, fragmentation still occurs but at bounded rates.   If $T^N$ is the instant  of creation of the first  polymer whose size is $n_c$,  it is shown that, under appropriate conditions,  the variable $T^N\!\!{/}[\Phi(N)^{n_c\!{-}\!2}\!\!{/}N]$ converges in distribution,  and that the first instant $L_\delta^N$ when a fraction $\delta$ of monomers is polymerized has the same order of magnitude. An original feature proved for this model is the significant variability of the variable $T^N$. This is a well known phenomenon observed in biological  experiments but few mathematical models of the literature have this property.   The results are proved via  a series of  technical estimates for occupation measures of some functionals of the corresponding Markov processes on fast time scales and by using coupling techniques.
\end{abstract}

\maketitle

\bigskip
\hrule

\vspace{-3mm}

\tableofcontents

\vspace{-10mm}

\hrule

\bigskip

\section{Introduction}
Polymerization  is an important phenomenon occurring in many areas, in particular for several biological processes.  In a biological context, some species of {\em proteins}, also called {\em monomers}, may be assembled via chemical reactions into aggregated states called {\em polymers}.  A polymer is an assembly of  proteins linked by some chemical bonds.  The {\em size/mass} of a polymer is the number of monomers composing it.  The random fluctuations  generated by the thermal noise within the biological cell are at the origin of the encounters of  polymers and monomers  which may lead to the growth of existing polymers.  These fluctuations may also break some chemical bonds within polymers and thus give polymers with  smaller sizes.  From an experimental point of view, the randomness of fluctuations seems to have an important impact on the time evolution of  polymerization processes. This is the main motivation of this  paper which investigates a stochastic model describing the time evolution of a set of polymers. 

Using classical notations for chemical reactions, for $k{\ge}1$, ${\cal X}_k$ represents a polymer of size $k$,  the polymerization process  analyzed in this paper can be represented as
\begin{equation}\label{ChimPol}
    \begin{cases}
  {\cal X}_1{+}{\cal X}_k {\overset{\kappa_{\rm{on}}^k}{\longrightarrow}} {\cal X}_{k+1},\\
\phantom{  {\cal X}_1{+} }{\cal X}_{k} {\overset{\kappa_{\rm{off}}^k}{\longrightarrow}}{\cal X}_{a_1}{+}{\cal X}_{a_2}{+}\cdots{+}{\cal X}_{a_p}, \quad  p{\geq}2, a_1{+}a_2+\cdots{+}a_p=k.
 \end{cases}
\end{equation}
The quantity $\kappa_{\rm{on}}^k$ [resp. $\kappa_{\rm{off}}^k$] is the chemical rate at which a polymer of size $k$ is bound with a monomer [resp. at which it is broken]. See Appendix~\ref{Ap1} for a more detailed description of the biological background. It should be noted that polymers can only grow by the addition of a monomer. This is in fact a common assumption in this domain as long as linear polymers, also called {\em fibrils} are considered.  The fragmentation mechanism is, a priori, arbitrary. 

\subsection*{The Nucleation Phenomenon}
We introduce the notion of {\em nucleation} which is a standard assumption in the corresponding biological literature but, to the best of our knowledge, does not seem to have been considered in the previous stochastic mathematical models of polymerization.   See Morris et al.~\cite{Morris}, Kashchiev~\cite{Kashchiev2}.  The assumption is that  polymers with small sizes are quite unstable.  They are very quickly broken by the random fluctuations of the environment. There is nevertheless a critical size $n_c$ above which a polymer is more stable. It can be still broken due to random events but at a much smaller rate.  The quantity $n_c$ is called the {\em nucleus size}, in this way polymerization can also be seen as a nucleation process as in the literature in physics. These assumptions are generally based on considerations of statistical mechanics  expressed in terms of the free energy of assemblies of proteins.  See Kashchiev~\cite{Kashchiev3}. With a slight abuse of notation we will say that nucleation has occurred when a polymer of size $n_c$ has been created for the first time. 

\subsection*{The Variability of Polymerization Processes}
When starting with only monomers in experiments, the system  stays for some time with a negligible polymerized mass, i.e.\ with few  polymers. The main  phenomenon observed is that there is some instant when a sharp phase transition occurs,  the polymerized mass goes then  from $0$ to reach very quickly its final value. This explosive behavior is a constant feature of  polymerization processes observed in experiments. Another  key  property from a biological point of view is that the instant when this event occurs, the {\em lag time}, varies significantly from an experiment to another. See Radford~\cite{Radford} and Figure~\ref{fig1} in the appendix. This property is believed to be at the origin of the long delays observed in the development of  several neuro-degenerative diseases such as the Bovine Spongiform Encephalopathy  (Mad Cow) or Alzheimer's disease for example.   See Appendix~\ref{Ap1} for a quick review and references for the biological aspects. The main goal of mathematical analysis in this domain is of providing a simple model exhibiting this phenomenon.

\subsection*{Central Limit vs Rare Events  Theorems}
An analysis of the fluctuations has been achieved in the literature in terms of  a central limit theorem (CLT) around this sigmoid function with $N$, the total mass, as a scaling parameter.  See Szavits et al.~\cite{Szavits}, Eug\`ene et al.~\cite{EXRD}, see also Doumic et al.~\cite{DER}.  These references generally consider only two species of macro-molecules, monomers  and polymerized monomers. Unfortunately, when compared with experimental data,  these mathematical models fail in general to  explain the order of magnitude of the variability observed in the experiments. The main reason, as it is natural in a classical CLT setting, is that the variance is of the lag time of the order of $\sqrt{N}$, and therefore  much smaller than its average value, of the order of $N$. This is not what is observed in the experiments. See Figure~\ref{fig1} in the appendix for example. For small volumes, estimations obtained in this way can nevertheless be reasonably accurate, see  Eug\`ene et al.~\cite{EXRD}. 

The variability can thus be hardly explained only by a central limit theorem.  The main result of this paper states that the fluctuations  are, more likely,  due to the occurrence of a set of rare events. This has been suggested in the biological literature, see Section~4~b) and~c) of Hofrichter~\cite{Hofrichter} for example, see also Yvinec et al~\cite{Yvinec}. But, to the best of our knowledge, it has never been established rigorously with a convenient mathematical model.  


\subsection*{A Scaling Assumption For the Stochastic Model}
In the chemical reactions~\eqref{ChimPol} the existence of a nucleus size is generally represented as follows: for a polymer of size $k{<}n_c$, the chemical rate of fragmentation is much larger that the chemical rate of association to another monomer, i.e. $\kappa_{\rm{off}}^k{\gg}\kappa_{\rm{on}}^k$. Otherwise, it is of the same order of magnitude  or much smaller.

A natural Markovian description involves a multi-dimensional state space, the state descriptor is given by a vector $u{=}(u_k)$ where, for $k{\ge}1$, $u_k$ is the number of polymers of size $k$. The growth of polymers of size $k$ is described by the interaction of the  $k$th coordinate and of the first coordinate,  $u_k$ and $u_1$. The fragmentation of a polymer  is a more intricate transition since a polymer can be  fragmented into a subset of polymers of smaller sizes.  In our Markovian model, this is translated in the following way, recall that  $N$ is the total mass of the system, in state $u{=}(u_k)$,
\begin{enumerate}
\item[(I)] {\sc Growth.} A monomer is added to a given polymer of size $k$ at rate
  \[
  \lambda_k \frac{u_1}{N},
  \]
for some positive constant $\lambda_k$.  Three coordinates of the state change at this occasion: $u_1{\to}u_1{-}1$, $u_k{\to}u_k{-}1$ and $u_{k{+}1}{\to}u_{k{+}1}{+}1$. See for example  Anderson and Kurtz~\cite{AndersonKurtz} for a general presentation of mathematical models of chemical reactions in a stochastic context. 
\item[(II)] {\sc Fragmentation.}  The fragmentation of a given polymer of size $k$ occurs at rate $\mu_k^N$, with
  \[
  \mu_k^N\steq{def}
  \begin{cases}
    \Phi(N)\mu_k&\text{ if } k{<}n_c,\\
    \mu_{n_c} & \phantom{ if  } \ k{\ge}n_c,
  \end{cases}
  \]
for some positive constants $\mu_k$, $1{\le}k{\leq}n_c$.  The way the polymer is fragmented is described by a {\em fragmentation measure}. See Section~\ref{Sec0}.
\end{enumerate}
  See Figure~\ref{fig2}.
Small polymers disappear at rate proportional to $\Phi(N)$, where $\Phi$ is some non-decreasing function converging to infinity. In our approach $N$,  the scaling parameter,  can be also thought as a volume. With this interpretation, the quantity $u_1/N$ used in  the rate of growth above is the concentration of ``free'' monomers, i.e. polymers of size $1$.  See van Kampen~\cite{vanKampen} or Yvinec et al.~\cite{yvinec2016}. The function $\Phi$ can be taken as $\Phi(x){=}x^\gamma$, with $\gamma{\in}(0,1]$. 

 \subsection*{The Main Results}
 Assuming that that $N$ is the mass of the system and that the initial state does not contain any polymer of size greater or equal to $n_c$ and consists mainly of monomers,  if  $T^N$, the first nucleation time, is the first instant when a polymer of size $n_c$ is created,   we prove that, under appropriate conditions, the convergence in distribution holds
 \begin{equation}\label{eqlim}
 \lim_{N\to+\infty} \frac{N}{\Phi(N)^{n_c-2}} T^N=E_{\overline{\rho}},
\end{equation}
where $E_{\overline{\rho}}$ is an exponential random variable with parameter $\overline{\rho}$, a constant depending on  the constants $\lambda_k$, $\mu_k$, $k{<}n_c$, the polymerization and fragmentation rates of polymers of sizes less than $n_{c}{-}1$.  See Relation\eqref{rho} for an explicit expression. In particular,  the variable $T^N$  has a significant variability, it variance being of the same order of magnitude in $N$ as its average value.  The variable $L_\delta^N$ , the {\em lag time}, is defined as  the first instant when a fraction $\delta$ of  monomers is polymerized into polymers of size greater than $n_c$. We show that, in the limit and under appropriate conditions, $L_\delta^N$ has the same order of magnitude in $N$ as $T^N$\!\!.  See  Theorem~\ref{theolag}.

The mathematical analysis is decomposed into two steps.
 \begin{enumerate}
 \item[1.] The evolution of the process until a polymer of size $n_c$ is created. With our scaling assumptions, this is a rare event as mentioned above, from this point of view the convergence~\eqref{eqlim} is natural in probability theory. For stochastic processes converging quickly to equilibrium, rare events are, generally, reached after an exponentially distributed amount of time with a large average. See Keilson~\cite{Keilson} and Aldous~\cite{Aldous2} for example.  Nevertheless, getting an explicit expression for the asymptotic exponential random variable and establishing the corresponding technical estimates  turn out to be a quite challenging problem. See Section~\ref{OneSec} and  Appendix.
  \item[2.] With an initial state with only one polymer of size $n_c$, if this polymer  grows sufficiently quickly, its fragmentation will give several stable polymers that will also consume monomers and therefore may generate other stable polymers. It turns out that this is a very fast way to create stable polymers.  It is shown in fact that, with this initial state, the number of stable polymers can be stochastically lower-bounded by a super critical branching process, in particular with positive probability it is growing exponentially fast or it dies out. This is the explosion phase of the polymerization process. 
 \end{enumerate}
 Our model has thus the two main characteristics observed in the experiments in biology: a take-off phase with a large variability, the corresponding exponential distribution,  and a steep growth, the super-critical branching process.  

 \medskip
 \noindent
     {\bf Remarks}
 \begin{enumerate}
\item[(i)]  Relation~\eqref{eqlim} is proved under quite general assumptions on the way a polymer of size $k{\ge}2$ is fragmented. In the stochastic model it is represented by a probability distribution $\nu_k$ on the space of all possible decompositions of the integer $k$. See Section~\ref{Sec0}.  The limiting distribution of Relation~\eqref{eqlim} does not depend in fact on $\nu_k$ but only on fragmentation and growth rates. The intuitive, non-rigorous, reason is the following:  the first polymer of size $n_c$ is built from successive additions of monomers and the key step is in fact the one to have the ``last'' monomer added to a polymer of size $n_c{-}1$. If it is fragmented before then, very quickly, it is reduced to  a set of $n_c{-}1$ monomers. Of course this is a rough description. The precise result is established in Section~\ref{OneSec}.
\item[(ii)] When $\Phi(N){=}N^\gamma$, as our results show, this phenomenon occurs when the nucleus size satisfies the relation $n_c{>}2{+}1/\gamma$. For the current mathematical literature, the nucleus size used is (implicitly sometimes)  $2$, which may explain that this phenomenon has not been yet established rigorously. 
\item[(iii)] The results are obtained via a series of (quite)  technical estimates for occupations measures of the corresponding Markov process on the  fast time scale $t{\mapsto}\Psi(N){\cdot}t{\steq{def}}\Phi(N)^{n_c{-}2}/N{\cdot}t$.  Stochastic calculus with Poisson processes,  coupling arguments and  branching processes  are the main ingredients of the proofs. It should be noted that one of the main difficulties while using stochastic calculus is of controlling the fluctuations of the associated martingales,  on the a priori very fast time scale $(\Psi(N)t)$.  This is where a coupling with a strongly ergodic process plays a major role. See the proof of Proposition~\ref{lem2}. 
     \end{enumerate}

\subsection*{Literature}
There is a past and recent interest in {\em growth-fragmentation models} which are generalizations of the polymerization process described by Relation~\eqref{ChimPol}. In the context of coagulation and fragmentation models, the Smoluchowski model is a classical mathematical model. See Aldous~\cite{Aldous} for a general survey.  For these models, the growth can occur also by coagulation and not only by addition of a single particle/monomer at a time. Additionally, the rates of occurrence of the growth or fragmentation events can be state dependent.  These processes are also used to study some population processes. The studies focus mainly on the existence of such processes, on their scaling properties  (like  self-similarity) or their branching process representation.  A special and important case for polymerization processes is the {\em Becker-D\"oring model} for which the chemical reactions are given by 
 \begin{equation}\label{ChimBD}
  {\cal X}_1{+}{\cal X}_k \stackrel{{\overset{\kappa_{\rm{on}}^k}{\longrightarrow}}}{ \underset{\kappa_{\rm{off}}^{k+1}}{\longleftarrow}}  {\cal X}_{k+1}, \, k{\geq}1.
 \end{equation}
In particular the fragmentation mechanism is, in some way, degenerated since at most one monomer can be detached at a time from a polymer.    In a deterministic setting, there is an associated system of ODEs  called the {\em Becker-D\"oring Dynamical System} $(c_k(t))$ solution of the system of differential equations
 \begin{equation}\tag{BD}
   \begin{cases}
\displaystyle     \dfrac{  \diff c_1}{\diff t} (t)={-}2J_{1}(c(t)){-}\sum_{k\geq2} J_{k}(c(t)), \\ \dfrac{  \diff c_k}{\diff t} (t)=J_{k-1}(c(t)){-}J_{k}(c(t)), \, k{>}1,
   \end{cases}
\end{equation}
 with  $J_k(c){=}\kappa_{\rm{on}}^kc_1c_k{-}\kappa_{\rm{off}}^{k+1}c_{k+1}$ if $c{=}(c_k){\in}\R_+^\N$ and with a convenient initial condition. This is a first order description of the process, the term $c_k(t)$ should be thought of the concentration of polymers of size $k{\geq}1$. See Becker and D\"oring~\cite{Becker} for the original paper.  The conditions of existence and uniqueness of solutions have been extensively investigated. See Ball et al.~\cite{Ball}  for example. 
 
 In the literature of mathematical models of this biological phenomenon it has been extensively used, see the various classes of  ODEs in Table~2 of Morris et al.~\cite{Morris}. With convenient parameters estimations, these ODEs can describe first order characteristics such as the mean concentration of polymers with a given size. See  Prigent et al.~\cite{Prigent} for example.   Due to their deterministic formulation, they cannot really be used to investigate the fluctuations of the polymerization process. 

In a stochastic context,  Jeon~\cite{Jeon} shows for the Smoluchowski model, that with Poisson processes governing the dynamics of the transitions of binary coagulation and fragmentation, then, under appropriate conditions,  a convergence result holds for the coordinates properly scaled  and the limit is the solution of a set of deterministic ODE's. For this model the corresponding functional central limit theorem has been proved  in Sun~\cite{Sun}. See  Szavits et al.~\cite{Szavits}, Eug\`ene et al.~\cite{EXRD} and Doumic et al.~\cite{DER} for related stochastic models.  The Becker-D\"oring model itself cannot be a convenient model to describe the variability of  polymerization processes. The fragmentation mechanism of the model, the removal of monomers one by one from a polymer cannot give an appropriate explosive  behavior as observed in practice. On this matter, Condition~A{-}3 of Assumption~${\rm A}^*$ of our model below is key to get this property.

\subsection*{Outline of the Paper}
Section~\ref{Sec0} introduces the  stochastic model used to investigate the polymerization process, together with the notations and definitions used throughout this paper. In  Section~\ref{OneSec} limiting results are proved for  the first instant when there is a stable polymer, see Theorem~\ref{TheoPois} and Proposition~\ref{hitprop}.   Section~\ref{LagSec} considers the dynamic of the polymerization  for the polymers whose sizes are above the level of nucleation $n_c$.  The main results for the asymptotic behavior of the distribution  of the lag time are then proved, see Theorem~\ref{theolag}. The appendix gives a quick overview of some biological aspects of these processes. 
\subsection*{Acknowledgments} The authors are grateful to Marie Doumic and Wei-Feng Xue for numerous conversations  on these topics. They have been very helpful for our understanding of these quite delicate phenomena. 
\section{Stochastic Model}\label{Sec0}
In this section we introduce the stochastic model describing the polymerization processes with a nucleation phenomenon. For this model, polymers of small sizes are unstable and quickly break into polymers with smaller sizes. 

The state of the model is represented by a vector $u{=}(u_i)$, where $u_i$ is the number of polymers of size $i$. 
For $p{\geq}1$ we define 
  \[
{\cal S}_p\steq{def}\left\{(u_i){\in}\N^{\N_+}:  u_1{+}2u_2{+}\cdots{+}\ell  u_\ell {+}\cdots{=}p\right\},
\]
the state space of the process is thus given by  ${\cal S_N}$.  Since for $u{\in}{\cal S}_p$, all components with index strictly greater than $p$ are null, we will occasionally use the slight abuse of notation $u{=}(u_1,\ldots,u_p)$ in this situation. We denote by
\begin{equation}\label{Sinfty}
 {\cal S}_\infty{=}\bigcup_{k\geq 0} {\cal S}_k,
\end{equation}
the set of states with finite mass. For $p{\in}\N$, $e_p$ will denote the $p$th unit vector of ${\cal S}_\infty$.

One starts from an initial state  with total mass $N$ and consisting mainly in monomers. It will be also interpreted as a volume parameter.  When the system is  in state $u{=}(u_k){\in}{\cal S}_N$ then, for $k{\geq}1$, the quantity $u_k/N$ is defined as the concentration of polymers of size $k$. In the following, $N$ is used as a scaling parameter.

In a deterministic context the law of mass action gives the Michaelis Menten's kinetics equations for the concentration of the various polymers. See Chapter~6 of Murray~\cite{Murray}. 

  \begin{center}
  \begin{figure}[ht]
\begin{tikzpicture}[->,node distance=1.4cm, thick]
     \node[state] (A) [below] {\ ${\scriptscriptstyle {\red 1}}$\ \  };
     \node[none] (AA) [above=3mm]{\ } (A);
     \node[state] (B) [right of=A] {\ ${\scriptscriptstyle {\red 2}}$\ \  };
     \node[none] (BB) [right of=AA] {\ }  ;
     \node[state] (C) [right of=B] {\ ${\scriptscriptstyle {\red 3}}$\ \  };
     \node[none] (CC) [right of=BB]{\ } (C);
     \node[none] (C0) [right of=C] {\ldots};    
     \node[none] (00) [right of=CC] {\ };    
     \node[none] (11) [right of=00] {\ };    
     \node[none] (22) [right of=11] {\ };    
     \node[none] (33) [right of=22] {\ };    
     \node[none] (44) [right of=33] {\ };    
     \node[state] (C1) [right of=C0] {${\scriptscriptstyle {\red n_c\!{-}\!2}}$};
     \node[state] (D) [right of=C1] {${\scriptscriptstyle {\red n_c\!{-}\!1}}$};
     \node[state] (E) [right of=D] {$\ \scriptstyle{{ n_c}}\ $};
     \node[state] (F) [right of=E] {$\scriptscriptstyle{{  n_c\!{+}\!1}}$};
     \node[none] (G) [right of=F] {\ldots};
     
     \path
     (A) edge[tran,bend right]  node[below]  {$\scriptstyle{[\lambda u]_1\frac{u_1}{{\red N}}}$} (B)
     (B) edge[tran,bend right]  node[above]  {$\scriptstyle{[\mu u]_2{\red \Phi(\!N\!)}}$} (A)

     (B)    edge[tran,bend right] node[below]  {$\scriptstyle{[\lambda u]_2\frac{u_1}{{\red N}}}$} (C)
     (C) edge[dashed,tran,bend right] node[above=1mm]  {$\scriptstyle{[\mu u]_3}{\red \Phi(\!N\!)}$}   (BB)

     (C) edge[tran,bend right]  node[below]  {$\scriptstyle{[\lambda u]_3\frac{u_1}{{\red N}}}$} (C0)
     (C0)   edge[dashed, tran,bend right] node[above=1mm]  {$\scriptstyle{[\mu u]_{4}{\red \Phi(\!N\!)}}$} (CC)

     (C0) edge[tran,bend right]  node[below]  {\ } (C1)
     (C1)   edge[dashed,tran,bend right] node[above]  { } (00)

     (C1) edge[tran,bend right]  node[below]  {\hspace{-2mm}$\scriptstyle{[\lambda u]_{n_c\!{-}\!2}\frac{u_1}{{\red N}}}$} (D)
     (D)   edge[dashed,tran,bend right] node[above=1mm]  {\hspace{-8mm} $\scriptstyle{[\mu u]_{n_c\!{-}\! 1}{\red \Phi(\!N\!)}}$} (11)
     
     (D) edge[tran,bend right]  node[below]  {$\scriptstyle{[\lambda u]_{n_c\!{-}\! 1}\frac{u_1}{{\red N}}}$} (E)
     (E)   edge[dashed,tran,bend right] node[above=1mm]  {$\scriptstyle{[\mu u]_{n_c}}$} (22)
     
     (E) edge[tran,bend right]  node[below]  {\hspace{2mm}$\scriptstyle{[\lambda u]_{n_c}\frac{u_1}{{\red N}}}$} (F)
     (F) edge[tran,bend right]  node[below]  {\hspace{5mm} $\scriptstyle{[\lambda u]_{n_c\!{+}\!1}\frac{u_1}{{\red N}}}$} (G)
     (F)   edge[dashed,tran, bend right] node[above=1mm]  {\hspace{-5mm} $\scriptstyle{[\mu u]_{n_c\!{+}\! 1}}$} (33)
     (G)   edge[dashed,tran,bend right] node[above=1mm]  {\hspace{5mm}  $\scriptstyle{[\mu u]_{n_c\!{+}\! 2}}$} (44);

\end{tikzpicture}
  \caption{Transition rates of jumps between coordinates of $(u_k)$ with the notation $[\xi u]_k{=}\xi_ku_k$.  The interactions with the first coordinate and detailed fragmentation components (dashed lines) are not represented.}\label{fig2}
  \end{figure}
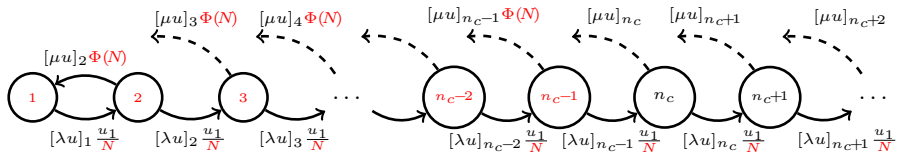
\end{center}

\subsection*{Growth of Polymers} The growth of polymers occurs only through  successive additions of  monomers. The rate at which  a given polymer of size $k$, $k{<}n_c$, is aggregated  with a polymer of size $1$ is taken as proportional  to the concentration $u_1/N$ of these polymers.  See Anderson and Kurtz~\cite{AndersonKurtz}.  Hence the total rate of production of polymers of size $k{+}1$ via this kind of reaction is given by
\[
\lambda_k u_k\frac{u_1}{N},
\]
for some constant $\lambda_k{>}0$.

\subsection*{Fragmentation}\label{FracSubSec} This is where the nucleation phenomenon is introduced in the stochastic model. It is assumed that $n_c{\in}\N$ such that  $n_c{>}2$. For $k{\geq}2$, the fragmentation rate of a polymer of size $k$ is given by
\begin{equation}\label{muas}
  \mu_k^N{=}
  \begin{cases}
    \Phi(N)\mu_k &\text{ if }k{<}n_c,\\
    \mu_{n_c}& \text{ if  } k{\geq}n_c,
  \end{cases}
\end{equation}
where $x{\mapsto}\Phi(x)$ is a positive non-decreasing function on $\R_+$ converging to infinity at infinity and  $\mu_k$, $2{\le}k{\le}n_c$, are positive real numbers.  In the following polymers whose size is greater than $n_c$ will be qualified as {\em stable}  to indicate that their fragmentation rate is not large. Stable polymers are fragmented at a constant rate $\mu_{n_c}$. 

In the literature for chemical reactions of polymerization processes, the ratio of growth rate and fragmentation rate, $\kappa^k_{\rm off}/\kappa^k_{\rm on}$ in our previous notations,  is assumed to be large for $k{<}n_c$ and small otherwise. See Kashchiev~\cite{Kashchiev2}.  In our case the scaling parameter $\Phi(N)$ stresses the difference of the dynamical behavior of polymers with size less than $n_c$.

In this setting, a quite general fragmentation process is considered. A polymer of size $k$  splits at rate $\mu_k^N$  according to a fragmentation distribution $\nu_k$ on the state space ${\cal S}_k$. More precisely, a polymer of size $k$ is broken into $Y^k_i$ polymers with size $i$, $1{\le}i{\le}k{-}1$  where $Y^k{=}(Y^k_1,\ldots,Y^k_{k-1})$ are random variables such that
\[
\E\left(f(Y^k_1,\ldots,Y^k_{k-1})\right)=\int_{{\cal S}_k} f(y_1,\ldots,y_{k-1})\,\nu_k(\diff y).
\]
For $p{<}k$, we denote by $I_p(y)$ the $p$th coordinate of $y{\in}{\cal S}_k$ and
\[
\croc{\nu_k,I_p}\steq{def}\int_{{\cal S}_k} I_p(y)\,\nu_k(\diff y)=\int_{{\cal S}_k} y_p\,\nu_k(\diff y)
\]
is the average number of polymers of size $p$ for $\nu_k$, in particular
\[
\sum_{p=1}^{k-1} p\croc{\nu_k,I_p}=k.
\]
\subsection*{Remark on the Nucleus Size}
If the nucleus size is $2$, in state $u{\in}{\cal S}_N$ stable polymers are created at rate $\lambda_1 u_1^2/N$. Hence, initially, stable polymers are created directly from monomers at a rate a the order of $\lambda_1N$: a significant fraction of monomers is polymerized right away. When $n_c{\ge}3$, it will be seen that, under some conditions,  stable polymers are essentially produced as follows:  a stable polymer grows for some time and then is fragmented into multiple polymers and, with positive probability, the size of several of them may be larger than $n_c$. A stable polymer can create stable polymers with positive probability. It turns out that this production scheme is much faster than the creation of stable polymers by the successive addition of monomers to polymers with size less than $n_c{-}1$. For this reason the case $n_c{=}2$ stands out. It is in fact used in most of the mathematical models of polymerization, implicitly sometimes. In our view,  this is the reason  why it cannot be really  used to explain the large variability observed in the experiments. 

\bigskip
\noindent
{\sc Examples of Fragmentation Measures.}\\
Fragmentation can be seen to a decomposition of integers. An important literature is in fact devoted to this topic: from the point of view of combinatorics as well as for statistical aspects. See, for example, Fristed~\cite{Fristed}, Pitman~\cite{Pitman} or Ercolani et al.~\cite{Ercolani}. We now give some classical examples.

\noindent
Recall that, for $k{\geq}2$,  a polymer of size $k$ is split according to the distribution $\nu_k(\cdot)$.
\begin{enumerate}
\item ${\rm UF}$: Uniform Binary Fragmentation, for $0{<}\ell{<}k$,
    \[
    \nu_k^{{\rm UF}}(e_{\ell}{+}e_{k-\ell})=
    \begin{cases}
      {2}/{(k{-}1)},\quad &  k\ {\rm odd},\\
      {2}/{(k{-}1)}, \quad \ell{\not=}k/2&  k\  {\rm even},\\
      {1}/{(k{-}1)}, \quad \ell{=}k/2&  k\  {\rm even},
    \end{cases}
    \]
    where $e_{\ell}$ is the $\ell$th unit vector of $\N^\N$, $e_\ell$ is representing a single polymer of size $\ell$.
  \item[]
  \item  ${\rm BF}$: Binomial Fragmentation, $p\in(0,1)$, if $k$ is odd and $0{<}\ell{<}k$,
      \[
      \nu_{k}^{\rm BF}(e_{\ell}{+}e_{k-\ell})=\frac{2}{1{-}p^k{-}(1{-}p)^k}\binom{k}{\ell} p^\ell(1{-}p)^{k-\ell}, 
      \]
     If $k$ is even, the case $\ell{=}k/2$ has to be singled out as before. 
\item[]
\item  ${\rm MF}$: Multiple Fragmentations.\\
 For $m{\ge}2$ and $p{=}(p_i){\in}(0,1)^{m}$ with $p_1{+}p_2{+}\cdots{+}p_m{=}1$, a polymer of size $k{\ge}m$ is fragmented into $m$ polymers with sizes $k_1{+}1$,  $k_2{+}1$, \ldots, $k_m{+}1$ according to a multinomial distribution,  if $\ell{=}e_{k_1{+}1}{+}e_{k_2{+}1}{+}\cdots{+}e_{k_m{+}1}{\in}{\cal S}_k$
    \[
      \nu_{k,m}^{{\rm MF}}(\ell)=\sum_{\substack{(n_1,\ldots,n_m){\in}\N^m\\\{n_1,\ldots,n_m\}=\{k_1,\ldots,k_m\}}}\frac{(k{-}m)!}{n_1!n_2!\cdots n_m!}\prod_{i=1}^m p_i^{n_i}. 
      \]
      if $k{<}m$, one set $\nu_{k,m}^{{\rm MF}}(\ell){=}1$ for $\ell{=}k{\cdot}e_1$.
\end{enumerate}
The family of probability distributions to describe fragmentation have various definitions depending on the context investigated.  For continuous fragmentation, i.e.\  when the state space is $\R_+$ instead of $\N$, the corresponding quantity is the {\em dislocation measure}. In the  self-similar case, the fragmentation process is described by a measure associated with the outcome of the breaking of a particle of size $1$. There are generalizations in this continuous setting with fragmentation kernels $K(x,\diff y)$ depending on the initial size $x$ to decompose.  See Bertoin~\cite{Bertoin1} for example.  For fragmentation of integers, as it is our case, Mohamed and Robert~\cite{Mohamed:01} describes the fragmentation also via a single measure, {\em the splitting measure}  for the analysis of first order quantities.  See~\cite{Mohamed:01} for an overview of this literature in this framework.

We introduce the following conditions under which our main results are established.
\begin{assumption}
  \begin{enumerate}
    \item[]
  \item[A{-}1)] Lower Bound for Polymerization and Fragmentation Rates. 
\begin{equation}\label{lambdab}
    \underline{\lambda}\steq{def}\inf_{k{\geq 1}}\lambda_k{>}0 \text{ and }\inf_{k{\le}n_c} \mu_k {>}0.
\end{equation}
\item[A{-}2)]  Scaling function $\Phi$ for Fragmentation . It is assumed that  
\begin{equation}\label{Psi}
  \Psi(N)\steq{def}\frac{1}{N}\Phi(N)^{{n_c}{-}2}
  \quad\text{ is such that }\quad
  \lim_{N\to+\infty} \frac{\Psi(N)}{\log N} =+\infty,  
\end{equation}
and, for any $\eta{\ge}0$, $(N/\Phi(N)^\eta)$ is a converging sequence with a limit in $\R_+{\cup}\{{+}\infty\}$.
We define
\begin{equation}\label{kc}
  k_c\steq{def}\sup\left\{k{\in}\N_+:\lim_{N\to+\infty} \frac{N}{\Phi(N)^{k-1}}={+}\infty\right\},
\end{equation}
note that, by the assumption on $\Psi$, $k_c{\le}n_c{-}2$.

\item[A{-}3)]  Fragmentation of Polymers. There exists $C_0{>}0$ and $K$ such that, for $k{\ge}K$,
\begin{equation}\label{eqA31}
  \nu_k\left( y{=}(y_i){\in}{\cal S}_k{:} \sum_{i<n_c}y_i\leq C_0\right)=1, 
\end{equation}
and  
\begin{equation}\label{eqA32}
\liminf_{k\to+\infty} \nu_k\left( y{=}(y_i){\in}{\cal S}_k{:} \sum_{i\geq n_c}y_i\geq 2\right)>0.
\end{equation}
\item[A{-}4)] Monotonicity Property: If $k{\leq}{k'}$, then, for all $a{\in}\N$ and $\ell{\ge}1$, the relation
\[
\nu_k\left(y{\in}{\cal S}_k{:}  \sum_{i=1}^{+\infty}\ind{y_i\geq \ell } \ge a\right)     \leq \nu_{k'}\left(y{\in}{\cal S}_{k'}{:} \sum_{i=1}^{+\infty}\ind{y_i\geq \ell } \ge a\right)
\]
holds. 
\end{enumerate}
\end{assumption}
The condition on the convergence of $(N/\Phi(N)^\eta)$, $\eta{>}0$, is for technical convenience, mainly in the proof of Proposition~\ref{lem1}.

Assumption~A{-}3 states that, with probability $1$, the fragmentation of a large polymer can give at most $C_0$ unstable polymers, i.e.\ whose size is less than $n_c$. Under this condition, a large polymer cannot  be broken into monomers only. Consequently, the fragmentation of a polymer of size greater than $n_cC_0$ produces at least a polymer of size greater than $n_c$. Furthermore, with a positive probability, two stable polymers are produced from the fragmentation of such  a large  polymer. Note that the fragmentation process of the Becker-D\"oring model~\eqref{ChimBD} does not satisfy this condition. Recall, as mentioned in the introduction, that this mathematical model fails to exhibit an explosive growth behavior.  The monotonicity property~A{-}4 gives that the sizes of fragments are stochastically increasing with respect to the size of the fragmented polymer.

\begin{prop}
The fragmentation measures ${\rm UB}$, ${\rm UF}$ and ${\rm MB}$ satisfy Assumption~A${}^*$.
\end{prop}
\begin{proof}
  The proof is done for the multinomial fragmentation, the proof for the symmetrical fragmentation is similar. The fragmentation as the distribution of $k{-}m$ balls into $m$ urns so that the probability  the $i$th urn is chosen  is $p_i$, if $A_i^k$ the number of balls in this urn, then $\nu_k$ is simply the empirical distribution of these variables. A simple coupling can be constructed so that $A_i^k{\le}A_i^{k+1}$ holds for all $i{\in}\{1,\ldots,m\}$. This gives right away the monotonicity property. The law of large numbers gives the convergence in distribution
  \[
  \lim_{k\to+\infty} \frac{1}{k}(A_1^k,\ldots,A_m^k)= (p_1,\ldots,p_m),
  \]
 Condition~A{-}3 holds since $p_i$ is positive for all $i$. 
\end{proof}

\subsection*{The Infinitesimal Generator of the Stochastic Evolution}
We define by $U_k^N(t)$ the number of polymers of size $k{\geq}1$ at time $t{\geq}0$.  Clearly, the process $(U^N(t)){\steq{def}}(U_k^N(t))$ has the Markov property on the countable state space ${\cal S}_N$.  
The generator $\Omega_N$ of this process is  given by, for $u{\in}{\cal S}_N$,
\begin{multline}\label{Gen}
  \Omega_N(f)(u)=  \sum_{k=1}^{+\infty}\lambda_k u_k \frac{u_1}{N} \left[f(u{+}e_{k+1}{-}e_k{-}e_1){-}f(u)\right]\\+
 \sum_{k=2}^{+\infty}  \mu_k^N u_k  \int_{{\cal S}_k} \left[f(u{+}y{-}e_k){-}f(u)\right]\nu_k(\diff y)
\end{multline}
where, for $i{\geq}1$, $e_i$ is the $i$th unit vector of $\N^{\N}$ and $f$ is a function on ${\cal S}_N$ with finite support. 

The growth mechanism involves the scaled term $u_1/N$, the concentration of monomers. This is not the case for the rates of fragmentation which do not depend on the concentration,  in the state $u{=}(u_i){\in}{\cal S}_N$,  one of the polymers of size $k$ is fragmented at rate $\mu_k^N u_k$.
\subsection*{Stochastic Differential Equations}
Throughout the paper, we use the following notations for Poisson  processes. For  $\xi{>}0$,  ${\cal N}_{\xi}$ is a Poisson point process on $\R_+$ with parameter $\xi$ and $({\cal N}_{\xi,i})$ is an i.i.d.\ sequence of such processes.

For  $k{\geq}2$, $\overline{{\cal N}}^k_{\xi}$ is a marked Poisson point process on $\R_+{\times}{\cal S}_k$ with intensity measure $\xi\diff t{\otimes}\nu_k(\diff y)$ and $(\overline{{\cal N}}^k_{\xi,i})$ is an i.i.d.\ sequence of such processes.
Such a process can be represented as follows.  If ${\cal N}_{\xi}{=}(t_n^k)$ is a Poisson process on $\R_+$ with rate $\xi$ and $(Y_n^k)$ is an i.i.d.\ sequence of random variables on ${\cal S}_k$ with distribution $\nu_k$, then $\overline{{\cal N}}^k_{\xi}$ has the same distribution as the point process on $\R_+{\times}{\cal S}_k$ given by 
\[
\sum_{n\geq 1}\delta_{(t_n^k,Y_n^k)},
\]
where $\delta_{(a,b)}$ is the Dirac mass at $(a,b)$.  See Kingman~\cite{Kingman} for an introduction on marked  Poisson point processes. All Poisson processes used are assumed to be independent.

In this context, the Markov process $(U^N(t))$ can also be seen as the stochastic process solution of the system of stochastic differential equations
\begin{multline}\label{SDE}
\diff U_k^N(t)= \sum_{\ell=1}^{U_1^NU_{k-1}^N(t-)}{\cal N}_{{\lambda_{k{-}1}}{/}{N},\ell}(\diff t)+\sum_{m=k+1}^{+\infty} \sum_{\ell=1}^{U_m^N(t-)} \int_{{\cal S}_m} y_k\overline{\cal N}^m_{\mu_m^N,\ell}(\diff t,\diff y)\\
-\sum_{\ell=1}^{U_k^N(t-)}\overline{\cal N}^k_{\mu_k^N,\ell}(\diff t,{\cal S}_k)-\sum_{\ell=1}^{U_1^NU_k^N(t-)}{\cal N}_{{\lambda_{k}}/{N},\ell}(\diff t), \quad k{\geq}2.
\end{multline}
 The evolution of $(U_1^N(t))$,  the process of the number of monomers,  is defined via   the conservation of  mass condition,
\begin{equation}\label{mass}
\sum_{k{\ge}1} kU_k^N\!(t)=\sum_{k{\ge}1} kU_k^N\!(0)=N.
\end{equation}

\medskip
\noindent
{\bf Condition for the Initial State}\\
For any $k{\ge}n_c$, $U_k^N(0){=}0$. For any $\eps{>}0$ and $2{\le}k{\le}k_c$, the relations
\begin{equation}\tag{I}\label{init}
  \begin{cases}
&\displaystyle \lim_{N\to+\infty} \frac{1}{N}U_1^N(0){=}1,\\
&\displaystyle \lim_{N\to+\infty} \frac{\Phi(N)^{k-1-\eps}}{N}U_k^N(0){=}0,  \\
&\displaystyle \lim_{N\to+\infty} \frac{1}{\Phi(N)^{\eps}} \left(U_{k_c+1}^N(0){+}\cdots{+}U_{n_c-1}^N(0)\right){=}0
  \end{cases}
\end{equation}
hold for the convergence in distribution. The function $\Psi$ is defined by Relation~\eqref{Psi} and $k_c$ by Relation~\eqref{kc}.

The motivation for this condition is that
\begin{itemize}
\item The limit result, Proposition~\ref{hitprop}, for the first instant when a stable polymer is created, holds for all these initial states.
\item Condition~(I) is in fact satisfied as long as few stable polymers have been created. This property will play an important role in Section~\ref{TwoSec}.
\end{itemize}
The initial state $(U^N\!(0)){=}(N,0,\ldots,0,\ldots)$ with only monomers, which is the classical initial condition used in the literature,  satisfies this property.  The process $(U^N(t))$ is c\`adl\`ag, i.e.\ right continuous with left limits at every point of $\R_+$, and $\diff U_k^N(t)$ is $U_k^N(t){-}U_k^N(t{-})$, where $f(t{-})$ denotes the left limit of a function $f$ at $t$. 

\subsection*{Nucleation Times}
We can now introduce the lag time of the polymerization process, for $\delta{\in}(0,1)$,
\begin{equation}\label{HitL}
  L_\delta^N=\inf\left\{t\geq 0:  \sum_{k=n_c}^{+\infty}k U_{k}^N(t){\ge} \delta N\right\}.
\end{equation}
This is the main quantity of interest in the paper, the first instant when there is a fraction $\delta$ of the mass (in terms of monomers) used in  stable polymers, i.e.\ whose sizes are greater than $n_c$.
The first nucleation time is defined as 
\begin{equation}\label{Hit1}
  T^N=\inf\{t\geq 0: U_{n_c}^N(t){=}1\},
\end{equation}
it is the first instant when a stable polymer is created.

\section{The First Instant of Nucleation}\label{OneSec}
The purpose of this section is of proving a convergence theorem for the first time when a polymer of size $n_c$ is created when the  initial mass $N$ converges to infinity. This hitting time  is of the order $(\Psi(N))$ defined by Relation~\eqref{Psi}. The proof is quite technical, we outline the general strategy to establish  Proposition~\ref{hitprop} which is the main convergence result.

\begin{enumerate}
\item An auxiliary $n_c$-dimensional  Markov process $(X^N(t))$ describing particles evolving state $1$ and state $n_c{-}1$ is introduced. It corresponds to the original process but truncated at the $(n_c{-}1)$th coordinate and with a different dynamic at node $n_c$.  For this process, the node $n_c$ is a cementary state for the particles.
\item This process is analyzed on a fast time scale  $t{\mapsto}\Psi(N)t$, Proposition~\ref{lem2} shows that the order of magnitude of coordinates with index greater or equal to $2$ are negligible with respect to $N$. A further result, Lemma~\ref{no} states that, still  on this time scale,  ``most'' of particles are in state $1$, in particular ``few'' particles are at node $n_c$.
\item Proposition~\ref{lem2} and Proposition~\ref{lembal} investigate the asymptotic behavior of occupation times of the Markov process $(X^N(t))$ of the form
  \[
  \int_0^{\Psi(N)t} f\left(X^N(s)\right)\,\diff s,
  \]
 where $f(\cdot)$ is a specific family of polynomial functions on $\R^{n_c}$. The key  convergence result of Proposition~\ref{lembal} can be interpreted as a kind of flow balance equation between the nodes $1$,\ldots, $n_c{-}1$  on the fast time scale $t{\mapsto}\Psi(N)t$. It gives, with an additional convenient martingale argument, the convergence in distribution of the sequence of instants  when particles arrive at node $n_c$ to a Poisson process,  this is Theorem~\ref{TheoPois}. The main result of this section, Proposition~\ref{hitprop}, is then a simple consequence of this theorem. The proof of Proposition~\ref{lembal} is quite involved, it uses Proposition~\ref{lem2} and several estimates.  The proof of Proposition~\ref{lem2}  has been put in the appendix to make the section more readable. 
\end{enumerate}

We define the Markov process $(X^N(t)){\steq{def}}(X_k^N(t))$, the solution of the system of stochastic differential equations
\begin{multline}\label{SDE1}
\diff X_k^N(t)= \sum_{\ell=1}^{X_1^NX_{k-1}^N(t-)}{\cal N}_{{\lambda_{k-1}}/{N},\ell}(\diff t)+\sum_{m=k+1}^{n_c-1} \sum_{\ell=1}^{X_m^N(t-)} \int_{{\cal S}_m} y_k\overline{\cal N}^m_{\mu_m\Phi(N),\ell}(\diff t,\diff y)\\
-\sum_{\ell=1}^{X_k^N(t-)}\overline{\cal N}^k_{\mu_k\Phi(N),\ell}(\diff t,{\cal S}_k)-\sum_{\ell=1}^{X_1^NX_k^N(t-)}{\cal N}_{{\lambda_{k}}/{N},\ell}(\diff t), \quad 2{\le}k{\leq}n_c-1,
\end{multline}
and
\begin{equation}\label{SDE1nc}
  \diff X_{n_c}^N(t)= \sum_{\ell=1}^{X_1^NX_{n_c-1}^N(t-)}{\cal N}_{{\lambda_{n_c-1}}/{N},\ell}(\diff t),
\end{equation}
with initial condition $(X^N(0))$ satisfies Condition~\eqref{init}.

As before, the mass conservation condition,
\[
\sum_{k=1}^{n_c}k X_k^N(t)=\sum_{k=1}^{n_c}k X_k^N(0)=N,
\]
on any finite time interval defines the evolution of the first coordinate $(X_1^N(t))$.
By comparing SDE's~\eqref{SDE} and~\eqref{SDE1}, we remark that the process $(X_k^N(t))$ is closely related to the polymerization process. The main difference is that  the $n_c$th coordinate is  a cemetery state.  With a slight abuse, for convenience,  for $1{\le}k{\le}n_c$ and $t{\ge}0$, we will refer to $X_k^N(t)$ as the number of polymers of size $k$ at time $t$. This process is used to investigate the first phase of the polymerization process, until the first nucleus is created. Another auxiliary process will be introduced in section~\ref{TwoSec} to investigate the second phase during which polymers of size greater than $n_c$ have a total mass of the order of $N$.

\medskip

\noindent
{\bf Remark.}\\
A related model (in a quite different context)  with a Becker-D\"oring flavor is analyzed in Sun et al.~\cite{SFR} with a different perspective since the goal is of analyzing the asymptotic behavior of a transient multi-dimensional  Markov process.  Outside the fragmentation feature which does not appear in the models of~\cite{SFR}, the transition rates are linear with respect to the state in  this reference, instead of a quadratic dependence for the present model.  In both cases, several estimates on fast time scales have nevertheless to be derived.

The strategy to derive the limiting behavior of  the distribution of the sequence of processes  $(X_k^N(\Psi(N)t))$, where $\Psi$ is defined by Relation~\eqref{Psi} is explained briefly. One will first prove, via quite technical estimates, that, on the fast time scale $t{\mapsto}\Psi(N)t$, the values of the coordinates with index between $2$ and $n_c{-}1$ are essentially negligible with respect to $N$. See Proposition~\ref{lem1}. The  second key result is the asymptotic balance equations~\eqref{eqbal} of Proposition~\ref{lembal},  giving estimations of occupation measures associated to the coordinates of the process on the time scale $t{\mapsto}\Psi(N)t$.  Theorem~\ref{TheoPois} is the main result of this section which establishes the convergence in distribution of the sequence $(X_{n_c}^N(\Psi(N)t))$ to an homogeneous Poisson process. A limit result for the asymptotic behavior of the first nucleation time $T^N$ defined by Relation~\eqref{Hit1} is derived from this theorem.

The integration of Equations~\eqref{SDE1} gives the following representation, for $1{<}k{<}n_c$,
\begin{multline}\label{SDE2}
X_k^N(t)=X_k^N(0)+\frac{\lambda_{k-1}}{N} \int_0^t X_1^N(s)X_{k-1}^N(s)\,\diff s+\sum_{\ell=k+1}^{n_c-1}  \mu_\ell\Phi(N)  \croc{\nu_\ell,I_k}  \int_0^t  X_\ell^N(s)\,\diff s\\
- \mu_k\Phi(N)\int_0^t  X_k^N(s)\,\diff s-\frac{\lambda_{k}}{N} \int_0^t X_1^N(s)X_{k}^N(s)\,\diff s+M_k^N(t),
\end{multline}
where $(M_k^N(t))$ is a martingale, obtained by the compensation of the Poisson processes of the dynamics. Stochastic calculus, see Section~\ref{mpp} of the Appendix for example,  gives that its previsible increasing process is 
\begin{multline}\label{SDEcroc}
\croc{M_k^N}(t)=
\frac{\lambda_{k-1}}{N} \int_0^t X_1^N(s)X_{k-1}^N(s)\,\diff s+\sum_{\ell=k+1}^{n_c-1}  \mu_\ell\Phi(N)  \croc{\nu_\ell,I_k^2} \int_0^t  X_\ell^N(s)\,\diff s\\
+ \mu_k\Phi(N)\int_0^t  X_k^N(s)\,\diff s+\frac{\lambda_{k}}{N} \int_0^t X_1^N(s)X_{k}^N(s)\,\diff s.
\end{multline}
For $k{=}n_c$, by definition the process $(X_{n_c}^N(t))$ is non-decreasing,
\begin{equation}\label{SDE3}
X_{n_c}^N(t)=\frac{\lambda_{n_c-1}}{N} \int_0^t X_1^N(s)X_{n_c-1}^N(s)\,\diff s+M_{n_c}^N(t),
\end{equation}
and
\begin{equation}\label{SDE3croc}
\croc{M_{n_c}^N}(t)=
\frac{\lambda_{n_c-1}}{N} \int_0^t X_1^N(s)X_{n_c-1}^N(s)\,\diff s.
\end{equation}

\medskip
By looking at the transition rates between the two first coordinates of $(X^N(t))$, one can guess that $(X_2^N(t))$ should be of the order of $N/\Phi(N)$ on the normal time scale at least. The following important proposition shows that, on the fast time scale $t{\mapsto}\Psi(N)\,t$, the coordinates with index $2$, \ldots, $n_c{-}1$ remain negligible with respect to $N$.

\begin{prop}\label{lem1}
If $(X_k^N(t))$ is the solution of the SDE~\eqref{SDE1} with initial state satisfying Condition~\eqref{init} then, for $2{\leq}k{\leq}n_c{-}1$ and any $\eps{>}0$, for the convergence in distribution of processes, the relations
\[
\begin{cases}
\displaystyle\lim_{N\to+\infty}\left(\frac{\Phi(N)^{k-1-\eps}}{N}X_k^N(\Psi(N)t)\right){=}(0),\quad 2{\le}k{\leq}k_c\\
\displaystyle \lim_{N\to+\infty} \left(\frac{1}{\Phi(N)^{\eps}} \left(X_{k_c+1}^N{+}\cdots{+} X_{n_c-1}^N\right)(\Psi(N)t)\right){=}(0).
\end{cases}
\]
hold.
\end{prop}
\begin{proof}
We fix $\eps{\in}(0,1/2)$, define ${{\lambda}}{\steq{def}}\max(\lambda_k,1{\leq}k{<}n_c)$,  ${\mu}{\steq{def}}\min(\mu_i/(i{-}1),2{\leq}i{\leq}n_c{-}1)$ and, for $k{=}2,\dots,n_c{-}1$,
\[
  \left(Z_k^N(t)\right)\steq{def}\left(\sum_{i=k}^{n_c-1} (i{-}1)X_i^N(t)\right).
\]
We start with the case $k{=}2$ process $(Z_2^N(t))$. Definition~\eqref{Gen} of the infinitesimal generator gives that, when  $X^N(0){=}x{\in}{\cal S}_N$, this process has positive jumps of size~$1$ only. Remember that a polymer of size $n_c$ is in a cemetery state from the point of view of the dynamic of the process  $(X^N(t))$. They occur  at rate
\[
\frac{x_1}{N} \sum_{i=1}^{n_c-2} \lambda_i x_i\leq \alpha_N^{(2)}\steq{def}\lambda N.
\]
Negative jumps of $(Z_2^N(t))$ are at least  less than ${-}1$ and occur at rate
\[
\frac{\lambda_{n_c-1}}{N}x_{n_c-1}x_1+\sum_{i=2}^{n_c-1} \Phi(N)\frac{\mu_i}{i{-}1} (i{-}1) x_{i}\geq {\delta_N} Z_2^N(t),
\]
with ${\delta_N}{\steq{def}}\mu\Phi(N)^{1-\eps}$.

With a simple coupling,  one can therefore construct   a birth and death process process  $(L_2^N(t))$, corresponding to an $M/M/\infty$ queue with arrival rate $\alpha_N^{(2)}$ and service rate $\delta_N$,  see Chapter~6 of Robert~\cite{Robert}, with $L_2^N(0){=}Z_2^N(0)$ and  such that  the relation $Z_2^N(t){\leq}L_2^N(t)$ holds for all $t{\geq}0$.  Let, for $n{\ge}0$,
\[
T_n^N=\inf\{s>0: L_2^N(s){\ge}n\},
\]
Relation~(6.13) of Proposition~6.9 in Robert~\cite{Robert} gives an expression for the Laplace transform of $T_n$, for $n{\in}\N$ and $\xi{\ge}0$,
\begin{multline}\label{expx1}
\E\left(e^{-{\delta_N}\xi T_n^N}\ind{Z_2^N(0)<n}\right)=\int_0^{+\infty} \E\left((1{+}r/\beta_N^{(2)})^{Z_2^N(0)}\ind{Z_2^N(0)<n}\right)r^{\xi-1}e^{-r}\,\diff r\\\left/\int_0^{+\infty} (1{+}r/\beta_N^{(2)})^nr^{\xi-1}e^{-r}\,\diff r,\right.
\end{multline}
with
\[
 {\beta_N^{(2)}}\steq{def}\frac{\alpha_N^{(2)}}{\delta_N}=\frac{\lambda}{\mu}\frac{N}{\Phi(N)^{1-\eps}}.
  \]
We assume for the moment that the limit of $(\beta_N^{(2)})$ is not $0$, in particular $k_c$ defined by Relation~\eqref{kc} is greater than $2$. See Assumption~A{-}2.

If  $\Gamma(\cdot)$ is the classical Gamma function,
\[
\Gamma(\xi)=\int_0^{+\infty}r^{\xi-1}e^{-r}\,\diff r, \quad \xi{>}0,
\]
by expanding the integrand of the Laplace transform by using the classical relation $\Gamma(m{+}1){=}m\Gamma(m)$, we get that
\begin{align}\label{expxi}
D_N(\xi){\steq{def}}\frac{1}{\Gamma(\xi)}\int_0^{+\infty}& (1{+}r/\beta_N^{(2)})^nr^{\xi-1}e^{-r}\,\diff r\\ &=\sum_{k=0}^n \binom{n}{k} \frac{1}{(\beta_N^{(2)})^k} \frac{\Gamma(\xi{+}k)}{\Gamma(\xi)}
=1+\sum_{k=1}^n \binom{n}{k} \frac{1}{(\beta_N^{(2)})^k} \prod_{i=0}^{k-1}(i{+}\xi)\notag\\
&=1+\xi \sum_{k=1}^n\frac{n!}{k(n{-}k)!}\frac{1}{(\beta_N^{(2)})^k} \prod_{i=1}^{k-1}(1{+}\xi/i).\notag
\end{align}
For $x{>}0$, we fix  $n{=}2m_N$  with $m_N{=}\lceil x\beta_N^{(2)}\Phi(N)^{\eps}\rceil{=}\lceil x(\lambda/\mu)N/\Phi(N)^{1{-}\eps}\rceil$ and denote
\[
U_N\steq{def}\frac{1}{\mu\Phi(N)^{1-\eps}\Psi(N)}\frac{(2m_N)!}{m_N{\cdot}m_N!}\frac{1}{(\beta_N^{(2)})^{m_N}},
\]
then, for $ a{\ge} 0$,  Equation~\eqref{expxi} gives $D_N(a/(\delta_N\Psi(N))){\ge}1{+}a U_N$.
Since
\begin{multline*}
  U_N\ge\frac{1}{\mu\Phi(N)^{1-\eps}\Psi(N)\beta_N^{(2)}}\frac{(2m_N)!}{m_N!}\frac{(x\Phi(N)^\eps)^{m_N-1}}{m_N^{m_N}}\\
  =  \frac{1}{\lambda\Phi(N)^{n_c-2}}\frac{(2m_N)!}{m_N!}\frac{(x\Phi(N)^\eps)^{m_N-1}}{m_N^{m_N}}
  \ge  \frac{1}{\lambda}x^{m_N-1}\Phi(N)^{\eps(m_N{-}1){-}n_c{+}2},
\end{multline*}
due to the assumption of $(\beta_N^{(2)})$, the sequence $(m_N)$ is converging to infinity, and the same property holds for $(U_N)$. 
We have proved that $(D_N(a/(\delta_N\Phi(N))))$ converges to infinity.

Condition~\eqref{init}  for the initial state  gives that, for any $\eta{>}0$ arbitrarily small,  for the convergence in distribution,
\[
\lim_{N\to+\infty} \frac{\Phi(N)^{1-\eta}}{N}Z_2^N(0)=0, \text{ so that }\lim_{N\to+\infty} \frac{\Phi(N)^{\eps/2}}{\beta^{(2)}_N}Z_2^N(0)=0.
\]
For $\xi{\in}\R_+$, by using Lebesgue's Theorem, we have the relation
\[
\lim_{N\to+\infty} \frac{1}{\Gamma(\xi)}\int_0^{+\infty} \E\left((1{+}x/\beta_N^{(2)})^{Z_2^N(0)}\ind{Z_2^N(0)<\beta_N^{(2)}/\Phi(N)^{\eps/2}}\right)x^{\xi-1}e^{-x}\,\diff x=1. %
\]
Relations~\eqref{expx1} and~\eqref{expxi} give that, for any $a{>}0$ and $x{>}0$,
\[
\lim_{N\to+\infty} \E\left(e^{-a{T_n^N}/{\Psi(N)}}\ind{Z_2^N(0)<\beta_N^{(2)}/\Phi(N)^{\eps/2}}\right)\le
\lim_{N\to+\infty} 1\!\!\left/\!\!D_N\!\!\left(\frac{a}{\delta_N\Phi(N)}\right)\right.{=}0.
\]
and
\begin{multline*}
\limsup_{N\to+\infty} \E\left(e^{-a{T_n^N}/{\Psi(N)}}\ind{n>Z_2^N(0)>\beta_N^{(2)}/\Phi(N)^{\eps/2}}\right)\\ \leq
\limsup_{N\to+\infty} \P\left(Z_2^N(0)>\beta_N^{(2)}/\Phi(N)^{\eps/2}\right){=}0,
\end{multline*}
which gives the relation
\[
\lim_{N\to+\infty}\E\left(\exp\left(-a \frac{T_{ xN{/}\Phi(N)^{1{-}\eps}}^N}{\Psi(N)}\right)\right)=0.
\]
Hence, for any $t{\ge}0$, the sequence $(\P(T_{ xN{/}\Phi(N)^{1{-}\eps}}^N{\leq} t\Psi(N)))$ is converging to $0$, equivalently, for $x{>}0$,
\begin{equation}\label{eqMM}
\lim_{N\to+\infty}\P\left(\frac{\Phi(N)^{1-\eps}}{N}\sup_{0\leq s\leq t\Psi(N)}L_2^N(s)\geq x  \right)=0.
\end{equation}
The coupling relation between $(L_2^N(t))$ and $(Z_2^N(t))$ gives that the last relation also holds  when $(L_2^N(t))$ is replaced by the process $(Z_2^N(t))$.

Assume now that  the limit of the sequence $(\beta_N^{(2)}){=}((\lambda/\mu) N/\Phi(N)^{1-\eps})$ is $0$. There exists some constant $C$ such that $\delta_N{\geq} CN$, for all $N{\geq}1$. Then, with a simple coupling, there exists an $M/M/\infty$ queue $(\widetilde{L}_2^N(t))$ with arrival rate $\lambda$ and service rate $C$ such that $L_2^N(t){\le}\widetilde{L}_2^N(N t)$ holds for all $t{\ge}0$.  Proposition~6.10 of Robert~\cite{Robert} shows that, for any $\eta{>}0$ and $x{>}0$,
\[
\lim_{N\to+\infty}\P\left(\frac{1}{\Phi(N)^\eta}\sup_{s\leq t\Phi(N)^{n_c-2}}\widetilde{L}_2^N(s)\geq x\right)=0.
\]
By using again the coupling, this relation also holds for $(Z_2^N(t))$. This completes the proof of our assertion for $k{=}2$.

Assume that the proposition holds up to $k{<}n_c{-}1$.  The process $(Z_{k+1}^N(t))$  has also positive jumps of size~$1$ only, occurring at rate
\[
\frac{x_1}{N} \sum_{i=k}^{n_c-2} \lambda_i x_i\leq\lambda \sum_{i=k}^{n_c-2}  x_i,
\]
when the process $(X^N(t))$ is in state $x{\in}{\cal S}_N$.

Assume by induction that the convergence holds for $k{<}n_c{-}1$. If $k{\le}k_c$, we therefore have that, for any $\eps{\in}(0,1/2)$ and $T{>}0$,  the relation
\[
\sup_{0\leq t\leq \Psi(N) T} \sum_{i=k}^{n_c-2}  X_i^N(t)\leq \frac{N}{\Phi(N)^{k-1-2\eps}}
\]
holds  with high probability for $N$ sufficiently large. As before, we introduce an $M/M/\infty$ process $(L_{k+1}^N(t))$ starting at $Z_{k+1}^N(0)$  with the service rate $\delta_N$ defined above,  and with the arrival rate $\alpha_N^{(k+1)}$  given by  $\lambda N/\Phi(N)^{k-1-2\eps}$, such that for any $\eta{>}0$, there exists some $N_0$ for which the relation
\[
\P\left(Z_{k+1}^N(t){\leq} L_{k+1}^N(t), \forall t{\in}(0,\Psi(N)T)\right)\geq 1{-}\eta
\]
holds for all $N{\ge}N_0$. Due to condition~\eqref{init}  for the initial state,  $(Z_{k+1}^N(0)/\beta_N^{(k+1)})$ converges in distribution to $0$, where $\beta_N^{(k+1)}{=}\alpha_N^{(k+1)}\!\!/\delta_N{=}(\lambda/\mu) N/\Phi(N)^{k-3\eps}$. We can now proceed with the same method  for $(L_{k+1}^N(t))$ as we did with $(L_2^N(t))$. The case $k{>}k_c$ is similar. The proposition is proved. \end{proof}

\begin{prop}\label{lem2}
For $n_c{\geq}r{\geq}2$, $1{\leq}k{\leq}((r{-}1){\land} (n_c{-}2))$ and $2{\leq}h{\leq}n_c{-}1$ then, for the convergence in distribution of continuous processes,
\begin{align}
\lim_{N\to\infty}\left(\frac{1}{N\Phi(N)^{r-2}}\int_0^{\Psi(N) t}X_{{n_c}-k}^N(u)X_h^N(u)\diff u\right){=}(0),\label{relation1}\\
\lim_{N\to\infty}\left(\frac{1}{N\Phi(N)^{r-1}}\int_0^{\Psi(N) t}X_{{n_c}-k}^N(u)X_1^N(u)\diff u\right){=}(0). \label{relation2}
\end{align}
\end{prop}
\begin{proof}
See Appendix. 
\end{proof}
We are now ready to establish one of the key technical results used in the proof of the main convergence theorem of this section. 
\begin{prop}[Balance Equations]\label{lembal}
  For $k{=}1$,\dots, $n_c{-}2$, if
\begin{multline*}
\left(\Delta_k^N(t)\right)\steq{def}\left({\frac{\lambda_{n_c-k-1}}{N}}\int_{0}^{t}X^N_1(u)^{k+1}X^N_{n_c-k-1}(u)\,\diff u\right.\\\left.{-}\mu_{n_c-k}\Phi(N) \int_{0}^{t}X^N_1(u)^k X_{n_c-k}^N(u)\,\diff u\right)
\end{multline*}
then, for the convergence in distribution,
\begin{equation}\label{eqbal}
\lim_{N\to+\infty}\left(\frac{1}{(N\Phi(N))^k}\Delta_k^N(\Psi(N)t)\right)=(0).
\end{equation}
\end{prop}
\begin{proof}
  By a careful use of the   SDE~\eqref{SDE1}, one gets that, for $t{\ge}0$ and $m{=}n_c{-}k$, where $1{\leq}k{\leq}n_c{-}2$
\begin{align*}
&X_1^N(t)^k X_{m}^N(t)={M}_{(k)}^N(t)\notag\\
&{+}{\frac{\lambda_{m-1}}{N}}\int_{0}^{t}X^N_1(u)^{k+1}X^N_{m-1}(u)\diff u
-\mu_m\Phi(N) \int_{0}^{t}X^N_1(u)^kX_{m}^N(u)\,\diff u \textcolor{blue}{\bigg\} \Delta_k^N(t)}\notag\\
&{+}{\frac{\lambda_{m-1}}{N}}\hspace{-1mm}\int_{0}^{t}\hspace{-1mm}\left(\left(X^N_1(u){-}1-\ind{m=1}\right)^{\mathclap k}{-}X^N_1(u)^k\right)\left(X^N_{m}(u){+}1\right)X^N_1(u)X^N_{m-1}(u)\diff u  \textcolor{blue}{\bigg\}A_1^N(t)}\notag\\
& {+}\mu_m\Phi(N) \int_{0}^{t}\int_{\cal{S}_{m}}\hspace{-2mm}\left(\left(X^N_1(u){+}y_1\right)^{\mathclap k}{-}X^N_1(u)^k\right)\left(X^N_{m}(u){-}1\right)\nu_{m}(\diff y) X_{m}^N(u)\,\diff u  \textcolor{blue}{\bigg\}A_2^N(t)}\notag\\
& {+}\sum_{i\neq m-1,m}\frac{\lambda_i}{N}\int_0^{t}\left(\left(X_1^N(u){-}1{-}\ind{i=1}\right)^{\mathclap k}{-}X_1^N(u)^k\right)X_{m}^N(u)X_1^N(u)X_i^N(u)\diff u \textcolor{blue}{\bigg\}A_3^N(t)}\notag\\
&{+}{\frac{\lambda_{m}}{N}}\int_{0}^{t}\left(\left(X^N_1(u){-}1\right)^{\mathclap k}\left(X^N_{m}(u){-}1\right){-}X^N_1(u)^kX^N_{m}(u)\right)X^N_1(u)X^N_{m}(u)\diff u \textcolor{blue}{\bigg\}A_4^N(t)}\notag\\
&{+}\sum_{i=2}^{m-1}\mu_i\Phi(N)\int_{0}^{t}\int_{\cal{S}_i}\left(\left(X^N_1(u){+}y_1\right)^{\mathclap k}{-}X^N_1(u)^k\right)X^N_{m}(u)X_i^N(u)\,\nu_{i}(\diff y) \diff u  \textcolor{blue}{\bigg\}A_5^N(t)}\notag \\
&{+}\hspace{-3mm}\sum_{i=m+1}^{{n_c}-1}\mu_i\Phi(N)\notag \\ &\hspace{5mm} \times \int_{0}^{t}\int_{{\cal S}_i}\hspace{-1mm}\left(\left(X^N_1(u){+}y_1\right)^{\mathclap k}\left(X^N_{m}(u){+}y_{m}\right){-}X^N_1(u)^kX^N_{m}(u)\hspace{-1mm}\right)\nu_{i}(\diff y) X_i^N(u)\diff u. \textcolor{blue}{\bigg\}A_6^N(t)}.\notag
\end{align*}
The terms $A_i^N(t)$, $i{=}1$,\ldots,$6$, are associated to the compensators of the following transitions, for $x{\in}{\cal S}_N$, 
\begin{itemize}
\item $A_1^N(t)$: it is for the transition $x{\to}x{-}e_1{-}e_{m-1}{+}e_m$ and with the first term of $\Delta_k^N(t)$
  \[
  {\frac{\lambda_{m-1}}{N}}\int_{0}^{t}X^N_1(u)^{k+1}X^N_{m-1}(u)\diff u
  \]
substracted. 
\item $A_2^N(t)$: transition when the fragmentation of an element of size $m$ gives $y_1$ polymers of size $1$ and with the second term of $\Delta_k^N(t)$
  \[
  -\mu_m\Phi(N) \int_{0}^{t}X^N_1(u)^kX_{m}^N(u)\,\diff u
  \]
substracted.
\item $A_3^N(t)$: transition $x{\to}x{-}e_1{-}e_{i}{+}e_{i+1}$ for $i{\not}{\in}\{m{-}1,m\}$.
\item $A_4^N(t)$: $x{\to}x{-}e_1{-}e_m{+}e_{m+1}$.
\item $A_5^N(t)$: the fragmentation of a polymer of size $2{\le}i{<}m$ gives $y_1$ polymers of size~$1$. 
\item $A_6^N(t)$: the fragmentation of a polymer of size $i{>}m$ gives $y_1$ polymers of size~$1$ and $y_m$ polymers of size~$m$.
\end{itemize}
Associated to these transitions,  $({M}_{(k)}^N(t))$ is the corresponding martingale. 
The above relation can thus be rewritten as
\begin{equation}\label{eqx12}
\Delta_k^N(t)=X_1^N(t)^k X_{n_c{-}k}^N(t)-{M}_{(k)}^N(t)-\sum_{i=1}^6 A_i^N(t).
\end{equation}
Proposition~\ref{lem1} shows that for  $T{\ge}0$, for $N$ sufficiently large, with high probability, the relation 
$X_{n_c{-}k}^N(\Psi(N)t){\le}N/\Phi(N)^{n_c{-}k{-}2}$ holds for all $t{\le}T$. 
Consequently the relation
\[
\frac{X_1^N(\Psi(N)t)^k X_{n_c{-}k}^N(\Psi(N)t)}{N^k\Phi(N)^k}\leq
\frac{X_{n_c{-}k}^N(\Psi(N)t)}{\Phi(N)^k}\le \frac{N}{\Phi(N)^{n_c{-}2}}
\]
also holds with high probability in the limit. In particular, due to Assumption~A{-}2, for the convergence in distribution
\[
\lim_{N\to+\infty} \left(\frac{X_1^N(\Psi(N)t)^k X_{n_c{-}k}^N(\Psi(N)t)}{N^k\Phi(N)^k}\right)=0.
\]
We now show that the processes $({A_i^N(\Psi(N)t)}/{(N^k\Phi(N)^k)})$, $i{=}1,\ldots,6$, also vanish as $N$ gets large. 
Since
\[
  A_1^N(t)={\frac{\lambda_{m-1}}{N}}\int_{0}^{t} B_1\left(X_1^N(u),X_m^N(u)\right)X^N_1(u)X^N_{m-1}(u)\,\diff u,
\]
  with $B_1(x,z){\steq{def}}\lambda_{m-1}((x{-}1)^k{-}x^k)(z{+}1){\le}c_1 x^{k-1}(z{+}1)$ for all $x$, $z{\in}\N$ and some convenient constant $c_1$, therefore we have
  \begin{multline*}
    \frac{A_1^N(\Psi(N)t)}{N^k\Phi(N)^k}\leq \frac{c_1}{N\Phi(N)^k} \int_{0}^{\Psi(N)t} \left(1{+}X_{n_c{-}k}^N(u)\right)X^N_{n_c{-}k{-}1}(u)\,\diff u\\\leq
    \frac{c_1}{N\Phi(N)^k} \int_{0}^{\Psi(N)t} \left( X_{n_c{-}k}^N(u)X^N_{n_c{-}k{-}1}(u) + X^N_{n_c{-}k{-}1}(u)^2\right)\,\diff u.
  \end{multline*}
  If $k{<}n_2{-}2$, Relation~\eqref{relation2} of Proposition~\ref{lem2} gives that the last process vanishes as $N$ gets large for the convergence in distribution.

  Otherwise, if $k{=}n_c{-}2$,
  \begin{multline*}
    \frac{A_1^N(\Psi(N)t)}{N^k\Phi(N)^k}\leq \frac{c_1}{N\Phi(N)^{n_c-2}} \int_{0}^{\Psi(N)t} \left( X_{1}^N(u)X^N_{2}(u) + X^N_{1}(u)\right)\,\diff u\\
\leq\frac{c_1}{N\Phi(N)^{n_c-2}} \int_{0}^{\Psi(N)t} X_{1}^N(u)X^N_{2}(u) \,\diff u
   +\frac{c_1\Psi(N)t}{\Phi(N)^{n_c-2}}.
  \end{multline*}
  The first term of the right hand side of the last relation is vanishing due to Relation~\eqref{relation2} of Proposition~\ref{lem2} and the second term is clearly converging to $0$.
  
Similarly, one can find a constant $c_2$ such that 
\[
A_2^N(t)\le c_2 \Phi(N) \int_{0}^{t}\hspace{-2mm}X^N_1(u)^{k-1}\left(X^N_{m}(u){-}1\right) X_{m}^N(u)\,\diff u,
\]
so that
\[
\frac{A_2^N(\Psi(N)t)}{N^k\Phi(N)^k}\le\frac{c_2}{N\Psi(N)^{k-1}}\int_{0}^{\Psi(N)t} \hspace{-4mm}|X_{m}^N(u){-}1|X^N_{m}(u)\,\diff u
\leq \frac{c_2}{N\Psi(N)^{k-1}}\int_{0}^{\Psi(N)t}\hspace{-4mm} X_{n_c-k}^N(u)^2\,\diff u,
\]
holds. Again  Proposition~\ref{lem2} can be used to show that this term vanishes as $N$ goes to infinity. 

Similarly, there are constants $c_i$, $i=3$,\ldots, $6$, such that the relations
\begin{align*}
\frac{A_3^N(\Psi(N))t}{N^k\Phi(N)^k}&\le\frac{c_3}{N\Phi(N)^{k}}\sum_{i\neq n_c-k-1,n_c-k}\int_{0}^{\Psi(N)t} X_{n_c-k}^N(u)X^N_{i}(u)\,\diff u,\\
\frac{A_4^N(\Psi(N))t}{N^k\Phi(N)^k}&\le\frac{c_4}{N\Phi(N)^{k}}\int_{0}^{\Psi(N)t} \left(X_{n_c-k}^N(u){+}X_1^N(u)  \right)X^N_{n_c-k}(u)\,\diff u,\\
\frac{A_5^N(\Psi(N))t}{N^k\Phi(N)^k}&\le\frac{c_5}{N\Phi(N)^{k-1}}\sum_{i=2}^{n_c-k-1}\int_{0}^{\Psi(N)t} X_{n_c-k}^N(u)X^N_{i}(u)\,\diff u,\\
\frac{A_6^N(\Psi(N))t}{N^k\Phi(N)^k}&\le\frac{c_6}{N\Phi(N)^{k-1}}\sum_{i=n_c-k+1}^{n_c-1}\int_{0}^{\Psi(N)t} \left(X_{n_c-k}^N(u){+}X_1^N(u){+}X^N_{i}(u)\right)X^N_{i}(u)\,\diff u
\end{align*}
hold and the corresponding terms also vanish at infinity for the convergence in distribution by using  Proposition~\ref{lem2}. 

To complete the proof, in view of Relation~\eqref{eqx12}, we have now to show that a similar result for the martingale $({M}_{(k)}^N(t))$, i.e. that, for the convergence in distribution,
\[
\lim_{N\to+\infty}\left(\frac{{M}_{(k)}^N}{N^k\Phi(N)^k}(\Psi(N)t)\right)=(0).
\]
From Relation~(3.31) of Lemma~I.3.30 of Jacod and Shiryaev~\cite{Jacod}, it is enough to show the following convergence in distribution of the sequence of previsible increasing processes,
\[
\lim_{N\to+\infty}\left(\croc{\frac{{M}_{(k)}^N}{N^k\Phi(N)^k}}(\Psi(N)t)\right)=0.
\]
Since $({M}_{(k)}^N(t))$ is a sum of  martingales $({M}_{(k)}^{i,N}(t))$,  $i{\in}\{1,\ldots,6\}$, associated to Poisson processes corresponding to the difference transitions. By orthogonality of these martingales, see Proposition~A.10 of Robert~\cite{Robert} for example, due to the independence of the Poisson processes, it is enough to show that the previsible increasing process of each of them vanish when $N$ gets large. We will detail only the case of  $({M}_{(k)}^{1,N}(t))$, the other martingales can be analyzed in a similar way. For $t{\ge}0$, with standard arguments, we have
\begin{multline*}
  \croc{\frac{{M}_{(k)}^{1,N}}{N^k\Phi(N)^k}}(\Psi(N)t)
  =\frac{\lambda_{m-1}}{N^{2k+1}\Phi(N)^{2k}}\\ \times \int_0^{\Psi(N)t}
  \left(\left(X^N_1(u){-}1\right)^{\mathclap k}(X^N_{m}(u){+}1){-}X^N_1(u)^kX^N_{m}(u)\right)^2X^N_1(u)X^N_{m-1}(u)\diff u,
\end{multline*}
with the same notations as before, we get that
\begin{multline*}
\croc{\frac{{M}_{(k)}^{1,N}}{N^k\Phi(N)^k}}(\Psi(N)t)
\leq  \frac{c_1^2}{N^2\Phi(N)^{2k}} \int_0^{\Psi(N)t} (X_1^N(u){+} X^N_{n_c-k}(u))^2X^N_{n_c-k-1}(u)\diff u\\
\leq  \frac{2c_1^2}{N\Phi(N)^{2k}} \int_0^{\Psi(N)t} (X_1^N{+} X^N_{n_c-k}(u))X^N_{n_c-k-1}(u)\diff u.
\end{multline*}
Again with  Proposition~\ref{lem2}, we get therefore that the previsible increasing process of $({M}_{(k)}^{1,N}(t))$ is converging in distribution to $0$ and, consequently the same result holds for the martingale.

By gathering these convergence results in Relation~\eqref{eqx12}, the proof of the proposition is completed. 
\end{proof}

For  $k{=}1$,\ldots, $n_c{-}2$, denote by $\rho_k{=}\lambda_k/\mu_k$ and $\overline{\rho}{=}\lambda_1\rho_2\cdots\rho_{n_c-1}$.
The above proposition shows that the following  sequence of processes
\begin{multline*}
\left(\lambda_{n_c{-}k{-}1}\prod_{i=n_c-k}^{n_c-1}\hspace{-2mm}{\rho_i}  \frac{1}{N^{k+1}\Phi(N)^{k}}\int_{0}^{\Psi(N)t}X^N_1(u)^{k+1}X^N_{n_c-k-1}(u)\,\diff u\right.
\\ \left. {-}\lambda_{n_c{-}k}\prod_{i=n_c-k+1}^{n_c-1}\hspace{-2mm}{\rho_i}\frac{1}{N^{k}\Phi(N)^{k-1}} \int_{0}^{\Psi(N)t}X^N_1(u)^k X_{n_c-k}^N(u)\,\diff u\right)
\end{multline*}
 converges in distribution to $0$. By summing up all these terms, we obtain the relation
\begin{multline}\label{balance}
\lim_{N\to+\infty}\left(\frac{\overline{\rho}}{N^{n_c-1}\Phi(N)^{n_c-2}}\int_{0}^{\Psi(N)t}X^N_1(u)^{n_c}\,\diff u\right.\\ \left. {-}\frac{\lambda_{n_c-1}}{N} \int_{0}^{\Psi(N)t}X^N_1(u) X_{n_c-1}^N(u)\,\diff u \right)=0.
\end{multline}

The next lemma establishes that, on the time scale $t{\mapsto} \Psi(N) t$, ``most'' of the polymers are still of size $1$.
\begin{lemma}\label{no}
The convergence in distribution 
 \begin{align*}
\lim_{N\to \infty}\left(\frac{1}{N}X^N\left(\Psi(N)\,t\right)\right)=(1,0,\ldots,0)
  \end{align*}
  holds. 
\end{lemma}
\begin{proof}
  The estimates of  Proposition~\ref{lem1} and the fact that the sum of the coordinates of $X^N(t)$ is $N$, show that we just have  to prove the convergence result for the last coordinate $(X_{n_c}^N(t))$.
  
The SDE~\eqref{SDE3} and Relation~\eqref{SDE3croc} give the identity
\begin{equation}\label{eqn0}
 X_{n_c}^N(t)=\lambda_{{n_c}-1} \int_0^{t}\frac{X_1^N(u)X_{{n_c}-1}^N(u)}{N}\diff u+M^N_{n_c}(t),
\end{equation}
where $(M_{n_c}^N(t))$ is a martingale.
By Relation~\eqref{balance}, for $t{\ge}0$, for the convergence in distribution, the sequence of processes
\begin{equation}\label{da1}
\left(\frac{\lambda_{{n_c}-1} }{N}\int_0^{\Psi(N)t}\frac{X_1^N(u)X_{{n_c}-1}^N(u)}{N}\diff u\right)
\end{equation}
has the same limit as
\[
\left(\overline{\rho}\frac{1}{N^{n_c}\Phi(N)^{n_c-2}}\int_{0}^{\Psi(N)t}X^N_1(u)^{n_c}\,\diff u\right)
=\left(\overline{\rho}\frac{1}{N}\int_{0}^{t}\left(\frac{X^N_1(\Psi(N)u)}{N}\right)^{n_c}\,\diff u\right),
\]
and, since the coordinates of the last process are upper bounded by $\overline{\rho}t/N$, we deduce the convergence to $(0)$ of the sequence of processes~\eqref{da1}.

The previsible increasing previsible process of the corresponding martingale is given by 
\[
\left(\croc{\frac{M_{n_c}^N}{N}}(\Psi(N)t)\right)=
\left(\frac{\lambda_{n_c-1}}{N^2} \int_0^{\Psi(N)t} \frac{X_1^N(s)X_{n_c-1}^N(s)}{N}\diff s\right).
\]
With the same argument, by using again Relation~(3.31) of Lemma~I.3.30 of Jacod and Shiryaev~\cite{Jacod}, we get that the sequence of martingales $(M_{n_c}^N(\Psi(N)t)/N)$ is converging in distribution to $0$. The lemma is proved. 
\end{proof}

We are now ready to establish the main result of this section. 
\begin{theorem}\label{TheoPois}
The sequence of processes $(X_{n_c}^N(\Psi(N)t)$ is converging in distribution to a Poisson process on $\R_+$ with rate
\begin{equation}\label{rho}
\overline{\rho}\steq{def}\lambda_1\prod_{k=2}^{{n_c}-1}\frac{\lambda_{k}}{\mu_k},
\end{equation}
where $\Psi(N){=}\Phi(N)^{n_c-2}/N$.
\end{theorem}
\begin{proof}
Clearly  $(X_{n_c}^N(\Psi(N)t))$ is a counting process, i.e.\  a non-decreasing integer valued process with jumps of size one. By Equation~\eqref{eqn0}, its compensator is given by
  \[
\left( \lambda_{{n_c}-1} \int_0^{\Psi(N)t}\frac{X_1^N(u)X_{{n_c}-1}^N(u)}{N}\diff u\right).
\]
By Relation~\eqref{balance}, this sequence of processes has the same limit as
\[
\left(\overline{\rho}\frac{1}{N^{n_c-1}\Phi(N)^{n_c-2}}\int_{0}^{\Psi(N)t}X^N_1(u)^{n_c}\,\diff u\right)
=\left(\overline{\rho}\int_{0}^{t}\left(\frac{X^N_1(\Psi(N)u)}{N}\right)^{n_c}\,\diff u\right).
\]
Lemma~\ref{no} shows therefore that the compensator of $(X_{n_c}^N(\Psi(N)t))$ is converging in distribution to $(\overline{\rho}t)$.  Theorem~5.1 of~Kasahara and Watanabe~\cite{Kasahara}, see also Brown~\cite{Brown}, gives the desired convergence in distribution to a Poisson process with rate $\overline{\rho}$.
The theorem is proved. 
\end{proof}

\begin{prop}[Asymptotic of the First Nucleation Time]\label{hitprop}
If  $((U^N(t))$ is the process defined by Relations~\eqref{SDE} and~\eqref{mass} and whose initial state satisfies Relation~\eqref{init}, and 
$T^N$\!\!,  defined by Relation~\eqref{Hit1}, is the first time its $n_c$th coordinate $(U_{n_c}(\cdot))$ is non null,  then, for the convergence in distribution,
  \[
  \lim_{N\to+\infty} \frac{N}{\Phi(N)^{n_c-2}}T^N=E_{\overline{\rho}},
  \]
$E_{\overline{\rho}}$ is an exponential random variable with parameter $\overline{\rho}$ defined by Relation~\eqref{rho}.
\end{prop}
\begin{proof}
  Let
  \[
  \tau^N\steq{def}\inf\{t{\ge}0: X_{n_c}^N(t){=}1\},
  \]
  then from the   SDE's~\eqref{SDE} and~\eqref{SDE1}, one gets the  identity
  \[
\left(\left(U_1^N,\ldots,U_{n_c}^N\right)(t{\wedge}T^N)\right)\steq{dist}\left(\left(X_1^N,\ldots,X_{n_c}^N\right)(t{\wedge}\tau^N)\right).
  \]
\end{proof}
\section{Asymptotic Growth of Lag Time}\label{LagSec}\label{TwoSec}
In this section, the  evolution of the polymerization process after nucleation is investigated. The main difference with Section~\ref{OneSec} lies in the fact that the polymers become more stable after nucleation: a polymer of size $k{\geq}n_c$ is degraded at rate $\mu_{n_c}$ instead of the rate $\mu_k \Phi(N)$ when $k{<}n_c$ as in Section~\ref{OneSec}. See Relation~\eqref{muas}. As before we first  introduce an auxiliary process to study this phase. 
\subsection{A Super-Critical Branching Process}
\addcontentsline{toc}{section}{ \hspace{5mm} A Super-Critical Branching Process}
We will be interested by the evolution of the number of polymers whose size are greater than  $n_c$. For $\alpha{>}0$ and $\mu{>}0$,  a Markov process $(Z^{\alpha,\mu}(t)){=}(Z^{\alpha,\mu}_{i}(t),i{\geq}n_c)$ will be introduced for this purpose.  With a slight abuse of notations, we will consider $(Z^{\alpha,\mu}(t))$ as a process in the state space ${\cal S}_\infty$. Formally, it can be done by assuming that the $n_c{-}1$ first coordinates are null.
We denote by $\|z\|_{c}$ the sum of the components of the vector $z{=}(z_i)$ of the set ${\cal S}_\infty$ of states with finite mass defined by Relation~\eqref{Sinfty},
\begin{equation}\label{nc}
\|z\|_{c}\steq{def}\sum_{i{\ge}n_c} z_i.
\end{equation}
The generator $\Omega_{Z^{\alpha,\mu}}$ of the process is given by 
\begin{multline}\label{Gen2}
  \Omega_{Z^{\alpha,\mu}}(f)(z)=\sum_{k=n_c}^{+\infty} \left[f(z{+}e_{k+1}{-}e_k){-}f(z)\right]  \alpha z_k\\+
  \sum_{k=n_c}^{+\infty} \int_{{\cal S}_k} \left[f(z{+}y{-}e_k){-}f(z)\right]\mu z_k\nu_k(\diff y)
\end{multline}
where, as before, for $i{\geq}1$, $e_i$ is the $i$th unit vector of ${\cal S}_\infty$ and $f$ is a function on ${\cal S}_\infty$ with finite support.
Since coordinates with index greater or equal to  $n_c$ are of interest, it is assumed that the function $f$ does not depend on the first $n_c{-}1$ coordinates. 

We give a quick, informal, motivation for the introduction of  the Markov process $(Z^{\alpha,\mu}(t))$. It describes in fact the evolution of stable polymers. Assume for the moment that the polymerization rates  are independent of the sizes of polymers, i.e.\  $\lambda_k{=}\alpha$ for all $k{\ge}n_c$. A monotonicity argument will be used to have a more general framework.  The initial state is assumed to be  $Z^{\alpha,\mu}(0){=}e_{n_c}$ with only one polymer of size $n_c$ present at time $0$, as it is the case just after the first nucleation instant. Proposition~\ref{no} gives that, at this instant, the number of polymers is small with respect to $N$, i.e.\  the fraction of monomer is close to $1$.  A polymer of size greater that $n_c$ grows therefore  at a rate close  $\alpha$ as in Relation~\eqref{Gen2}. If the fragmentation of a polymer of size $k{\ge}n_c$ gives a polymer of size ${<}n_c$, then, due to the fast fragmentation rates below $n_c$, this last one is fragmented quickly into monomers and thus vanishes as in Relation~\eqref{Gen2} since coordinates with index less than $n_c$ are assumed to be $0$ for the process $(Z^{\alpha,\mu}(t))$. 
\begin{prop}\label{propZ}
If $Z^{\alpha,\mu}(0){=}e_{m}$ for some $m{\ge}n_c$, there exist $\kappa_0{\ge}0$,  $a_0{>}0$ and $\eta{>}0$ such for any $\alpha$ and $\mu{>}0$  such that  $\alpha/\mu{\geq}\kappa_0$, then the event
\[
{\cal F}_{Z^{\alpha,\mu}}=\left\{  \liminf_{t\to+\infty}e^{-a_0 t}\|Z^{\alpha,\mu}(t)\|_c>\eta\right\}
\]
has a positive probability. 
\end{prop}
\begin{proof}
The evolution of $(Z^{\alpha,\mu}(t))$ describes the the population of stable polymers generated by an initial polymer with size $m{\ge}n_c$. 
 Relation~\eqref{eqA31} of  Condition~A{-}3 gives the existence of $\eps{>}0$ and $k_0{\ge}C_0n_c$ such that if  $k{\ge}k_0$, 
\[
\nu_k\left( y: \sum_{i\geq n_c}y_i{\geq} 2\right)  >\eps.
\]
A stable polymer of size $m{\ge}n_c$ is fragmented after an exponentially distributed amount of time  $E_\mu$ with parameter $\mu$. Just before this instant its size has the same distribution as $M{=}m{+}{\cal N}_\alpha([0,E_\mu])$. Recall that only the fragments with size greater than $n_c$ are considered for $(Z^{\alpha,\mu}(t))$.  The process $(Z^{\alpha,\mu}(t))$ can then be also seen as a  multi-type branching process with the type of an individual being the size of the corresponding polymer.    The average number of stable polymers generated by the fragmentation of the polymer of size $M$ is greater than 
\begin{multline*}
  \E\left(\nu_M\left( y: \sum_{i\geq n_c}y_i{\geq} 1\right)\right)+\E\left(\nu_M\left( y: \sum_{i\geq n_c}y_i{\geq} 2\right)\right)\\
  \geq \P(M{\ge}k_0)(1+\eps)=\left(1+\eps\right)\left(\frac{\alpha}{\alpha{+}\mu}\right)^{k_0-n_c}\hspace{-8mm}.
\end{multline*}
By choosing $\alpha/\mu$ sufficiently large, the last quantity is strictly greater than $1$. We have thus shown that the process $(\|Z^{\alpha,\mu}(t)\|_c)$ is lower bounded by a continuous time supercritical branching process. The proposition is then a simple consequence of a classical result in this domain, see Chapter~V of Athreya and Ney~\cite{Athreya} for example. 
\end{proof}

\subsection{Limit Results}
\addcontentsline{toc}{section}{\hspace{5mm} Limit Results for the Lag Time}
 now return to the original polymerization process $(U_k^N (t))$ with values in the state space ${\cal S}_N$ defined by the SDE~\eqref{SDE} and initial state satisfies Condition~\eqref{init}. We study the asymptotic behavior of  the associated lag time $L_\delta^N$ defined by Relation~\eqref{HitL}, which is the first time when the mass of stable polymers exceeds $\delta N$,
\[
L_\delta^N=\inf\left\{t\geq 0:  \sum_{k=n_c}^{+\infty}k U_{k}^N(t){\ge} \delta N\right\}
\]
The variable $T^N$  defined by Relation~\eqref{Hit1}, which is the first time when a stable polymer is created,
\[
T^N=\inf\left\{t\geq 0:   U_{n_c}^N(t){=}1\right\}
\]
is a stopping time such that  $T^N{\le}L_{\delta}^N$.
Let
\begin{equation}\label{Initn}
  u^{N}=(u^{N}_k)\steq{def} U^N(T^N),
\end{equation}
in particular  $u_{n_c}^N{=}1$  and $u_{k}^N{=}0$ for $k{>}n_c$. 

In the rest of this section, we denote by $(\widehat{U}^N(t))$ the solution of the SDE~\eqref{SDE} with initial condition~\eqref{Initn} and $\widehat{L}_{\delta}^N$ the corresponding lag time,
\[
\widehat{L}_{\delta}^N=\inf\left\{t\geq 0:  \sum_{k=n_c}^{+\infty}k \widehat{U}_{k}^N(t){\ge} \delta N\right\},
\]
note that, by the strong Markov property, $\widehat{L}_\delta^N{\steq{dist}}L_\delta^N{-}T^N$.
We have to study the order of magnitude of $\widehat{L}_\delta^N$. We will prove  that, with a positive probability, the mass of stable polymers hits the value $\lfloor \delta N\rfloor$ in a duration of time of the order of $\log N$.

The following proposition is the analogue of Proposition~\ref{lem1} but with the additional feature that there may be polymer of size greater than $n_c$. 
\begin{prop}\label{lem12}
For any $t_0$, $\eps{>}0$,   the convergence in distribution of processes
\begin{equation}\label{jg1}
  \begin{cases}
    \displaystyle   \lim_{N\to+\infty} \left(\frac{\Phi(N)^{k-1-\eps}}{N}\sum_{p=k}^{n_c-1}\widehat{U}^N_p(t), 0{\le}t{\le}  \min(\widehat{L}_{\delta}^N,\Psi(N)t_0)\right)=(0),\quad    2{\le}k{\le k_c,}\\
    \displaystyle \lim_{N\to+\infty} \left(\frac{1}{\Phi(N)^{\eps}} \left(\widehat{U}^N_{k_c+1}(t){+}\cdots{+} \widehat{U}^N_{n_c-1}(t)\right), 0{\le}t{\le}  \min(\widehat{L}_{\delta}^N,\Psi(N)t_0)\right){=}(0).
  \end{cases}
\end{equation}
hold. 
\end{prop}
\begin{proof}
It should be noted that, due to  Propositions~\ref{lem1} and~\ref{hitprop}, we have, for the convergence in distribution,
\[
\lim_{N\to+\infty}\frac{\Phi(N)^{1-\eps}}{N} \sum_{i=2}^{n_c-1} (i{-}1) \widehat{U}_i^N(0)=\lim_{N\to+\infty} \frac{\Phi(N)^{1-\eps}}{N}\sum_{i=2}^{n_c-1} (i{-}1){U}_i^N\left(T^N\right)=0.
  \]
We introduce the process
    \[
 \left(\widehat{B}_k^N(t)\right)\steq{def}\left(\sum_{i=k}^{n_c-1} (i{-}1)\widehat{U}_i^N(t)\right),
 \]
We proceed  as in the  proof of Proposition~\ref{lem1} for $k{=}2$. On the time interval $[0,\widehat{L}_{\delta}^N)$,  it can be stochastically  upper bounded by a jump process with initial state $\widehat{B}_2^N(t)(0)$, whose jump rates are given by, for $x{\in}\N$,
\begin{equation}\label{cpa}
 x\longrightarrow
 \begin{cases}
   x{+}1\text{ at rate } \lambda N\\
   x{-}1\phantom{ at rate aa} \mu_{n_c}\Phi(N) x \\
   x{+}(n_c{-}1)C_0\phantom{ at rate }\delta N\mu_{n_c},
 \end{cases}
\end{equation}
where  $\lambda$ is defined as in the proof of Proposition~\ref{lem1}. The last transition is associated to the fragmentation of polymers of size greater than $n_c$. Due to the time interval considered, the total  mass of these polymers is certainly less than $\delta N$ and the fragmentation of one of them gives at most $C_0$ polymers of size less than $n_c$ by Relation~\eqref{eqA31} of Assumption~A{-}3.

Assume that $k_c{\ge}2$. Let $a_0{=}(n_c{-}1)C_0$ and $\lambda_0{=}\lambda{+}\delta \mu_{n_c}$, then, on the time interval $[0,\widehat{L}_{\delta}^N)$,  the process $((\widehat{B}_2^N(t)))$ is stochastically dominated  by the jump process $(Q^N(t))$ with the same initial state and whose jump rates are given by, for $x{\in}\N$,
\begin{equation}\label{cpa2}
 x\longrightarrow
 \begin{cases}
   x{+}a_0\text{ at rate } \lambda_0 N\\
   x{-}1\phantom{ at rate aa} \mu\Phi(N)^{1-\eps/2} x.
 \end{cases}
\end{equation}
This is an $M/M/\infty$ with $a_0$ simultaneous arrivals. To prove the lemma, one has to use a similar argument as in the proof of Proposition~\ref{lem1}.
This is done as follows.  One can construct  a set of $a_0$ processes associated to $a_0$  $M/M/\infty$ queues, $(L_{2,\ell}^N(t))$, $\ell{\in}\{1,\ldots,a_0\}$. The queues share the same arrival process with  rate $\lambda_0N$ have independent services with  rate $\mu\Phi(N)^{1-\eps/2}$ and  initial conditions  $L_{2,1}^N(0){=}Q^N(0)$ and  $L_{2,1}^N(0){=}0$, for $\ell{=}2$,\ldots,$a_0$.  A coupling can be done so that $Q^N(t){=}L_{2,1}^N(t){+}L_{2,2}^N(t){+}\cdots{+}L_{2,a_0}^N(t)$ holds for all $t{>}0$. 

As in the proof of of Proposition~\ref{lem1},   for the convergence in distribution of processes, the relation
\[
\lim_{N\to+\infty} \left(\frac{\Phi(N)^{1-\eps}}{N}H(\Psi(N)t)\right)=(0)
\]
holds for all $H{=}L_{2,\ell}^N$,  $\ell{=}1$,\ldots,$a_0$ and, consequently, for $H{=}Q^N$. By domination, we finally get that 
\[
\lim_{N\to+\infty} \left(\frac{\Phi(N)^{1-\eps}}{N}\widehat{B}_2^N(\Psi(N)t)\right)=(0)
\]
holds.  The proof for $k_c{<}2$ and $2{<}k{\le}n_c{-}1$ follows the same kind of arguments as in the proof of Proposition~\ref{lem1}
\end{proof}

\begin{corollary}\label{lem22}
Under Assumption~A{-}3, 
for $\delta_0{>}\delta{>}0$  and $t_0{>}0$, if
\[
{\cal E}_N\steq{def} \left\{ \widehat{U}_1^N(t)\ge(1{-}\delta_0)N, \forall t\le\min(\widehat{L}_{\delta}^N,\Psi(N)t_0)\right\},
\]
then the sequence $(\P\left({\cal E}_N\right))$ is converging to $1$. 
\end{corollary}
\begin{proof}
The conservation of mass gives that, for all $t{\ge}0$,
\[
\sum_{i=1}^{+\infty} i\, \widehat{U}_i^N(t)=\sum_{i=1}^{+\infty} i\, \widehat{U}_i^N(0),
\]
and, noting that we have, for  $t{\le}\widehat{L}_{\delta}^N$,
\[
\sum_{i=n_c}^{+\infty} i\, \widehat{U}_i^N(t)\leq \lceil \delta N\rceil,
\]
we complete the proof of the lemma by using Relation~\eqref{jg1} with $k{=}2$ and $\eps{=}1/2$. 
\end{proof}

\subsection*{A Coupling}
Let us introduce a Markov  process $(Z^{\alpha_0,\mu_{n_c}}(t))$ with  the initial state $Z^{\alpha_0,\mu_{n_c}}(0){=}e_{n_c}$ and   the generator defined by Relation~\eqref{Gen2}  with $\alpha_0{\steq{def}}\underline{\lambda}(1{-}\delta_0)$, where $\underline{\lambda}$ is defined in Relation~\eqref{lambdab}.
\begin{prop}[Coupling]\label{Coup}
For $t_0{>}0$ and $\delta{\in}(0,\delta_0)$, under Assumptions~A${}^*$,  one can construct a coupling of the two processes $(Z^{\alpha_0,\mu_{n_c}}(t), t{\ge}0)$ and $(\widehat{U}^N(t), t{\ge}0)$ such that, on the event ${\cal E}_N$ of Corollary~\ref{lem22}, the relation 
\begin{equation}\label{eqcoupl}
  \sum_{k\ge n} Z^{\alpha_0,\mu_{n_c}}_k(t) \leq   \sum_{k\ge n} \widehat{U}_k^N(t), \quad \forall n{\geq}n_c,
\end{equation}
holds for all $0{\leq} t{\le}\min(\widehat{L}_\delta^N,\Psi(N)t_0)$.
\end{prop}
\begin{proof}
At time $0$, there is exactly a  polymer of size $n_c$ for $(\widehat{U}^N(t))$ and for $(Z^{\alpha_0,\mu_{n_c}}(t))$.  For $t{\ge}0$, recall, Relation~\eqref{nc},  that 
  \[
\|Z^{\alpha_0,\mu_{n_c}}(t)\|_c=\sum_{n{\geq}n_c}Z^{\alpha_0,\mu_{n_c}}_n(t)
  \]
  is the number of polymers for $(Z^{\alpha_0,\mu_{n_c}}(t))$.

If this last quantity is not $0$, we denote by $A_p(t)$, $1{\le}p{\le}\|Z^{\alpha_0,\mu_{n_c}}(t)\|_c$ the respective sizes of the corresponding polymers. The order of the sizes is arbitrary.  One will show that one can construct a coupling with  the following property: , on the event ${\cal E}_N$,  for  $0{\leq} t{\le}\min(\widehat{L}_\delta^N,\Psi(N)t_0)$, we can associate $\|Z^{\alpha_0,\mu_{n_c}}(t)\|_c$ distinct polymers described by the vector $\widehat{U}^N(t)$, whose sizes are given respectively by $B_p(t)$, $1{\le}p{\le}\|Z^{\alpha_0,\mu_{n_c}}(t)\|_c$,  and such that the relation $A_p(t){\leq}B_p(t)$ holds for all $1{\le}p{\le}\|Z^{\alpha_0,\mu_{n_c}}(t)\|_c$. This property implies that Relation~\eqref{eqcoupl} holds. It is proved by induction on the number of jumps of $(Z^{\alpha_0,\mu_{n_c}}(t))$ and $(\widehat{U}^N(t))$. It  clearly holds at time $0$. 

Assume that this relation holds at some fixed time $0{\leq} t{\le}\min(\widehat{L}_\delta^N,\Psi(N)t_0)$ , we will show that one can construct a version of the two process after that time so that the relation will also hold after the next jump of  $(Z^{\alpha_0,\mu_{n_c}}(t))$ and $(\widehat{U}^N(t))$. We now give the construction of the next jump.
\begin{enumerate}
\item A monomer addition to the polymer of size $A_i(t{-})$ occurs  at rate   $\alpha_0$. Remark that
  \[
  \alpha_0{\le}\lambda_{B_i(t{-})}\frac{\widehat{U}_1^N(t)}{N}.
  \]
Indeed, by definition,  $\lambda_{B_i(t{-})}{\ge}\underline{\lambda}$ and,  on the event ${\cal E}_N$, one has  the relation $\widehat{U}_1^N(t{-}){\ge}\underline{\lambda}(1{-}\delta_0)N$. 

The coupling is done so that a monomer addition to the polymer of size  $B_i(t)$ is also occurring  at that time for the process $(\widehat{U}^N(t))$. 
\item A monomer addition to the polymer of size $B_i(t{-})$ occurs  at rate  given by
  \[
  \lambda_{B_i(t{-})}\frac{\widehat{U}_1^N(t{-})}{N}{-}\alpha_0.
  \]
   There is no change for the process $(Z^{\alpha_0,\mu_{n_c}}(t))$ with this event.
\item If $A_i(t{-}){=}k$ and $B_i(t{-}){=}k'$, $n_c{\le}k{\le}k'$.\\
  At rate  $\mu_{n_c}$  both polymers of size $k$ [resp. of size $k'$]  are fragmented as $A^1_1$, \ldots $A^1_{n_A}$  [resp. $B^1_1$, \ldots $B^1_{n_B}$ ] according to the distribution $\nu_k$ [resp. $\nu_{k'}$].  By Assumption~A{-}4, the random variables $(A^1_i)$ and $(B^1_i)$ can be chosen so that, for any $1{\le}i{\le}n_A$, there exists some $1{\le}m_i{\le}n_B$ such that $A^1_i{\le}B^1_{m_i}$ and all indices $m_i$, $i={1}$, \ldots, $n_A$ are distinct. Note that we keep the $A^1_i$, $1{\le}i{\le}n_A$, whose values are greater or equal to $n_c$. 

\item All jumps involving the other polymers of $(\widehat{U}^N(t))$ are done as in the original setting. 
\end{enumerate}
With this construction, it is easily seen that  $(\widehat{U}^N(t))$ and $(Z^{\alpha_0,\mu_{n_c}}(t))$ are indeed Markov processes with the generator~\eqref{Gen}  and~\eqref{Gen2}  respectively. Moreover, each of the transitions described above preserve the desired relation.  The proposition is proved. 
\end{proof}
\begin{prop}\label{logprop}
Under Assumptions~${\rm A}^*$, if $\underline{\lambda}{>}\kappa_0\mu_{n_c}$, $\kappa_0$ is defined in Proposition~\ref{propZ} and $\underline{\lambda}$ by Relation~\eqref{lambdab}, then for any $\delta{<}1{-}\kappa_0\mu_{n_c}/\underline{\lambda}$, there exist  positive constant $p_0$ and  $K$, and $N_0$ such that, for any $N{\ge}N_0$,
\[
\P\left(\frac{\widehat{L}_{\delta}^N}{\log N} \leq K\right) \ge  p_0.
\]
\end{prop}
\begin{proof}
In the following, all statements are understood on the event ${\cal E}_N$ defined in Corollary~\ref{lem22}, this result shows  that this event has a probability close to $1$ as $N$ gets large. For $K{>}0$,  note that, due to Condition~\eqref{Psi} of Assumption~A{-}2, for $N$ sufficiently large then $K\log N{\le}t_0\Psi(N)$.
We fix $\delta_0{\in}(\delta,1{-}\kappa_0\mu_{n_c}/\underline{\lambda})$ and $\alpha_0{=}\underline{\lambda}(1{-}\delta_0)$.

From Proposition~\ref{Coup}, it can be assumed that there exists a Markov process $(Z^{\alpha_0,\mu_{n_c}}(t))$ with generator defined by Relation~\eqref{Gen2} and initial point $e_{n_c}$ such that the relation
  \[
\|Z^{\alpha_0,\mu_{n_c}}(t)\|_c=\sum_{k\ge n_c} Z^{\alpha_0,\mu_{n_c}}_k(t)\leq \sum_{k\ge n_c} \widehat{U}_k^N(t),
\]
holds for all $t{\le}\widehat{L}_{\delta}^N$. Since, for $t{\ge}0$,
\[
 \sum_{k\ge n_c} \widehat{U}_k^N(t)\le  \sum_{k\ge n_c} k\widehat{U}_k^N(t)
\]
\[
  {\cal H}_N\steq{def}\left\{\widehat{L}_{\delta}^N{>}K\log N\right\} \subset \left\{\|Z^{\alpha_0,\mu_{n_c}}(K\log N )\|_c\leq \lfloor {\delta} N\rfloor \right\},
  \]
since  $\underline{\lambda}(1{-}\delta_0){>}\mu\kappa_0$, with the notations of Proposition~\ref{propZ}, one can take $K{=}2/a_0$ then
\[
  {\cal F}_{Z^{\alpha_0,\mu_{n_c}}}{\cap} {\cal H}_N {=}\emptyset,
  \]
  as soon as $N{>} 2{\delta_0}/\eta$, where the event ${\cal F}_{Z^{\alpha_0,\mu_{n_c}}}$ is defined in Proposition~\ref{propZ}, hence
  \[
  \liminf_{N\to+\infty}   \P\left({\cal E}_N{\cap}{\cal F}_{Z^{\alpha_0,\mu_{n_c}}}{\cap}\left\{\widehat{L}_{\delta}^N \leq K\log N\right\}\right)=\P({\cal F}_{Z^{\alpha_0,\mu_{n_c}}}){>}0.
  \]
  The proposition is proved.
\end{proof}
\begin{theorem}[Growth Rate for the Lag Time]\label{theolag}
  Under Assumptions~A${}^*$,   if $\underline{\lambda}{>}\kappa_0\mu_{n_c}$,  the constant $\kappa_0$ is defined in Proposition~\ref{propZ} and $\underline{\lambda}$ by Relation~\eqref{lambdab}, then, if $\delta{\in}(0,1{-}\kappa_0\mu_{n_c}/\underline{\lambda})$ and for $\eps{>}0$, there exist $K_1$ and $K_2$ such that
\[
\liminf_{N\to+\infty} \P\left(K_1\leq \frac{L_\delta^N}{\Psi(N)} \leq K_2 \right)>1{-}\eps,
\]
where $\Psi(N){=}\Phi(N)^{n_c{-}2}/N$. 
\end{theorem}
\begin{proof}
The existence of the $K_1$ is a simple consequence of the fact that $L_\delta^N{\ge}T^N$ and Proposition~\ref{hitprop}.
 
For $i{\in}\N$, $E_{i,\overline{\rho}}$  will denote an exponentially distributed random variables with parameter $\overline{\rho}$  defined by Relation~\eqref{rho}.
  
By Proposition~\ref{hitprop} and the strong Markov property of $(U^N(t))$, at time $T_1^N{\steq{def}}T^N$ a polymer of size $n_c$ is created.  Proposition~\ref{hitprop} gives that $T_1^N/\Psi(N)$ converges in distribution to $E_{1,\overline{\rho}}$.   According to Proposition~\ref{logprop}, the exists some $N_0$ such that, for $N{\ge}N_0$, with probability at least $p_0$, there is a fraction $\delta$ of monomers is polymerized into stable polymers before time $R_1^N{\steq{def}}T^N{+}K\log N$.

If this does not happen, there are two possibilities:
  \begin{enumerate}
\item  There is at least one stable polymer at time $R_1^N$. One can construct another coupling with an independent process $(\widetilde{Z}^{\alpha_0,\mu_{n_c}}(t))$ with the same distribution as $(Z^{\alpha_0,\mu_{n_c}}(t))$. Again, there exists $N_1$ such that, if $N{\ge}N_1$,  with probability at least $p_0$ that  a fraction $\delta$ of monomers is polymerized into stable polymers before time $R_2^N{\steq{def}}R_1^N{+}K\log N$.  
\item There are no stable polymers at time $R_1^N$. Due to Proposition~\ref{lem12}, there exists some $N_1$ such that if $N{\ge}N_1$, the state $\widehat{U}^N(R_1^N)$ satisfies Condition~(I). If $R_1^N{+}T_2^N$ is the time of the first nucleation time after time $R_1^N$, Proposition~\ref{hitprop} gives that $T_2^N/\Psi(N)$ converges in distribution to $E_{2,\overline{\rho}}$.
\end{enumerate}
This decomposition shows that the lag time $L^N_\delta$ can be stochastically upper-bounded by a random variable
  \[
S^N\steq{def}\sum_{i=1}^{1+G_{p_0}} \left(T_i^N{+}K\log N\right),
  \]
where  $G_{p_0}$ is a geometrically distributed random variable with parameter $p_0$ independent of a sequence of random variables $(T_i^N)$ such that,  for $i{\ge}1$, the sequence $(T_i^N)$ converges in distribution to $E_{i,\overline{\rho}}$. It is easy to prove that the sequence  of random variables $(S^N/\Psi(N))$ is tight. The theorem is therefore proved. 
\end{proof}

\noindent
When $\Phi(N){=}N^{\gamma}$, $N{\ge}1$,
Assumption~A-2 requires that the nucleus size satisfies $n_c{>}2{+}1/\gamma$, the above theorem gives that the lag time is in this case of the order of $N^{\gamma(n_c{-}2){-}1}$.

\bibliographystyle{amsplain}
\bibliography{ref}
\newpage

\section{Appendix}\label{Ap1}
\subsection{Biological Background}
The protein polymerization processes  investigated in this paper are believed to be the main phenomena at the origin  of several  neuro-degenerative diseases such as Alzheimer's, Parkinson's and Huntington's diseases for example. The general picture of this setting is the following. At some moment,  for some reasons,    within  a neural cell  a fraction of the proteins of a given type  are produced in a anomalous state, defined as misfolded state. Recall that if a protein is a sequence of amino-acids,  its  three-dimensional structure determines also its functional properties. 

A misfolded protein has the same sequence of amino-acids but a different spatial architecture. It turns out that misfolded proteins tend to aggregate  to form fibrils, also called polymers. These fibrils are believed to have a toxic impact in the cell, on its membrane in particular, leading to its death.  The prion protein PrP$^{C}$ is an example of such protein that can be polymerized when it is in the state PrP$^{SC}$. The corresponding disease is the Bovine Spongiform Encephalopathy (BSE), also known as the mad cow disease.  This  (rough) description is not completely accurate  or complete, moreover some aspects are disputed, but it is used in a large part of the current literature. See the interesting historical survey Pujo-Menjouet~\cite{pujomenjouet}. Other biological processes such as actin filamentation, or yet industrial processes exhibit similar mechanisms,  see Wegner and Engel~\cite{Wegner}.

\subsubsection*{The Variability of the Polymerization Process}
Neuro-degenerative diseases are  quite diverse. They can be infectious, like the BSE, others are not, like Alzheimer (apparently). Nevertheless they all exhibit large, variable, time spans for the development of the disease, from several years to 10 years.

When experiments are done in vitro  with convenient types of proteins/monomers and with no initial polymers, a related phenomenon is observed. The fraction of monomers consumed by the polymerization process exhibit an $S$-curve behavior: it stays at $0$ for several hours, and quickly reaches $1$, the state where most of monomers are polymerized.  The other key feature of these experiments concerns the variability of the instant of the take-off phase of the $S$-curve from an experiment to another. See Szavits-Nossan et al.~\cite{Szavits} and  Xue et al.~\cite{Radford}.    See Figure~\ref{fig1} where twelve experiments are represented, the instant when half of the proteins are polymerized varies from 7.3h. to 10.1h. with an average of 8.75h.

\begin{figure}[ht] 
  \includegraphics[scale=0.35]{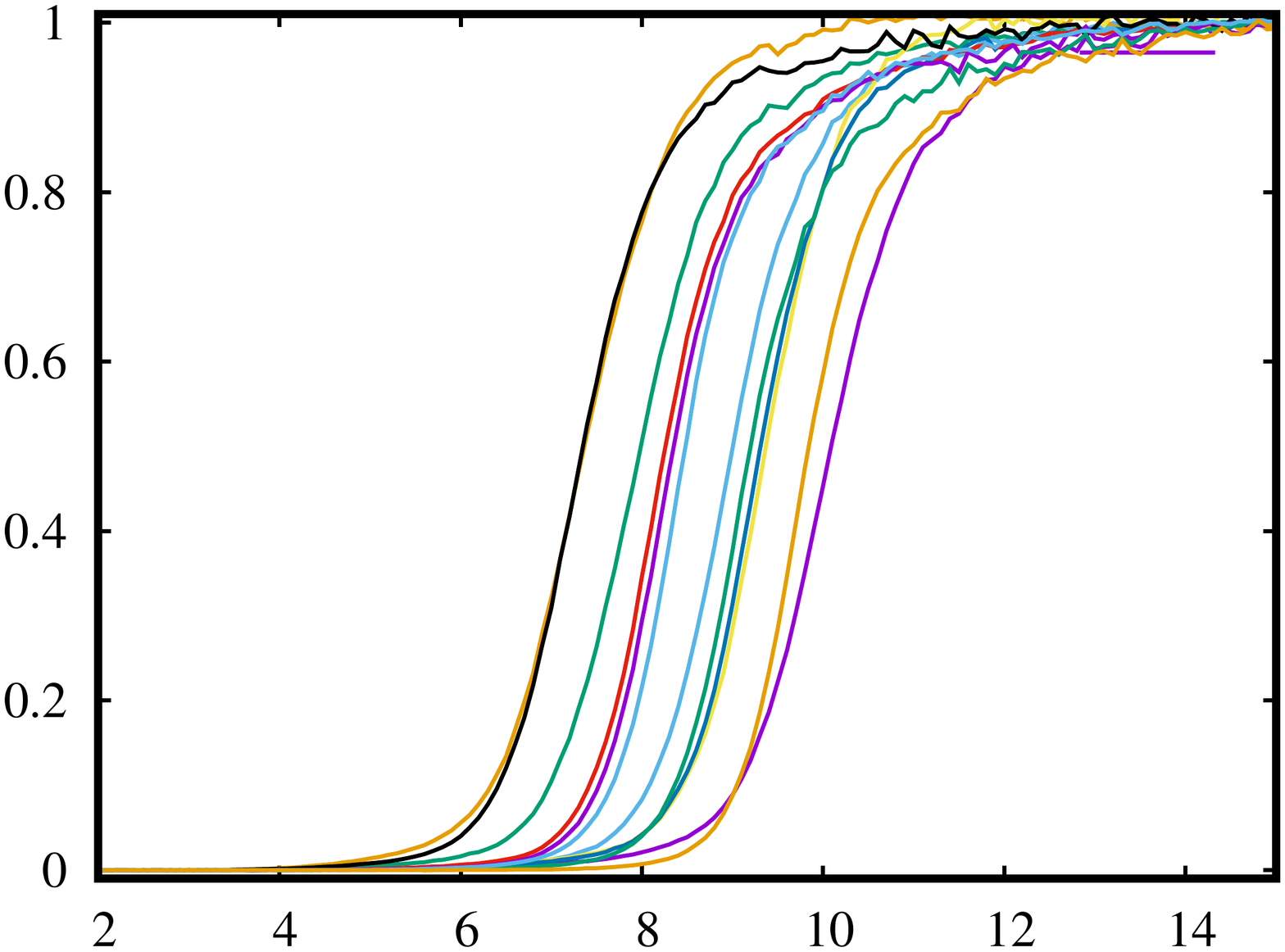}
  \put(-170,10){Time (hours)}
\put(-280,35){\rotatebox{90}{ Fraction of Polymers with Size ${\ge}2$}}
\caption{Twelve experiments for the time evolution of fraction of the mass of polymers with size greater $2$. From data published in Xue et al.~\cite{Radford}, see also Eugene et al.~\cite{EXRD}\label{fig1}.}
\end{figure}            
The initial step of the chain reactions giving rise to polymers  consists in the spontaneous formation of a so-called \emph{nucleus}, that is, the simplest possible polymer able to ignite the reaction of polymerization. This early phase is called \emph{nucleation}, and is still far from being understood.

\subsection{Marked Poisson Point Processes}\label{mpp}
We first recall briefly some elementary aspects of stochastic calculus with marked Poisson point processes. They are used throughout the paper.  See Jacobsen~\cite{Jacobsen} and Last and Brandt~\cite{Last} for more details. 
Let ${\cal N}_\lambda{=}(t_n)$ be a Poisson point process on $\R_+$ with parameter $\lambda$ and an independent  sequence  $(U_n)$ of i.i.d.  random variables on some locally compact space $H$, $\mu$ denotes the common distribution of these variables. The marked Poisson point process ${\cal N}_\lambda^U$ is defined as a point process on $\R_+{\times}H$, by
\[
{\cal N}_\lambda^U=\sum_{n\in\N}\delta_{(t_n,U_n)},
\]
if $f$ is a non-negative measurable function on $\R_+{\times}H$, one defines, for $t{\geq}0$,
\[
\int_0^t f(s,u){\cal N}_\lambda^U(\diff s, \diff u)=\sum_{n\in\N} f(t_n,U_n)\ind{t_n\leq t}
\]
and, if $F{\in}{\cal B}(H)$ is a Borelian subset of $H$,
\[
{\cal N}_\lambda^U([0,t]\times F)=\int_0^t \ind{u\in F}{\cal N}_\lambda^U(\diff s, \diff u)=\sum_{n\in\N} \ind{t_n\leq t, U_n\in F}.
\]
The natural filtration associated to ${\cal N}_\lambda^U$ is $({\cal F}_t)$, with, for $t{\geq}0$,
\[
{\cal F}_t=\sigma\left({\cal N}_\lambda^U([0,s]\times F): s{\leq} t, F{\in}{\cal B}(H)\right).
\]
\begin{prop}\label{propApp}
If $g$ is a c\`adl\`ag function on $\R_+$ and $h$ is Borelian on $H$ such that
  \[
\int_0^t g(s)^2\,\diff s{<}{+}\infty, \quad \forall t\geq 0 \text{ and } \int_H h(u)^2\nu(\diff u){<}{+}\infty,
\]
then the process
\[
(M(t))\steq{def}\left(\int_0^t g(s-)h(u){\cal N}_\lambda^U(\diff s, \diff u)-\lambda\int_H h(u)\nu(\diff u) \int_0^tg(s)\,\diff s \right)
  \]
  is a square integrable martingale with respect to the filtration $({\cal F}_t)$, its previsible increasing process is given by
  \[
  \left(\croc{M}(t)\right)=\left(\lambda\int_H h^2(u)\nu(\diff u) \int_0^tf(s)^2\,\diff s\right)
  \]
\end{prop}
In the paper, since we are dealing with several Poisson point processes, the (implicit)  definition of the filtration $({\cal F}_t)$ is extended so that it includes all of them. 

\subsection{Proof of Proposition~\ref{lem2}}\label{moresdes}
In this section we prove Relations~\eqref{relation1} and~\eqref{relation2} of the proposition by induction on $r$ varying from $n_c$ to $2$. The various stochastic integral equations used are listed in Section~\ref{SId} below.

\noindent
When $r{=}n_c$ and $k{\leq} n_c{-}2$, one has
\[
\left(\frac{1}{N\Phi(N)^{n_c-2}}\int_0^{\Psi(N)t}X_{{n_c}-k}^N(u)X_h^N(u)\diff u\right)\\{=}
\left(\int_0^{t}\frac{X_{{n_c}-k}^N}{N}{\cdot}\frac{X_h^N}{N}(\Psi(N)u)\,\diff u\right),
\]
Proposition~\ref{lem1} shows that this process converges in distribution to $0$ when $N$ goes to infinity  for all $k{=}1$,\dots,$n_c{-}2$ and $2{\leq}h{\leq}n_c{-}1$.
Relation~\eqref{relation2} also holds in this case.

Now suppose, by induction, that for all $r{>}\ell$ and $1{\leq}k{<}r{\land} (n_c{-}1)$, $2{\leq}h{\leq}n_c{-}1$, one has the convergences in distribution
\eqref{relation1} and \eqref{relation2}.
On will prove that this property holds for $r{=}\ell$ and for $k$ from $1$ to $(r{-}1){\land} (n_c{-}2)$.
Take $k{=}1$ and we will first prove the convergence~\eqref{relation1} for $h{=}n_c{-}1$ and Convergence~\eqref{relation2}, then Convergence~\eqref{relation1} by induction on $h$, from $n_c{-}1$ to $2$.

Take $k{=}1$ and $h{=}n_c{-}1$. Since $X_j^N(t){\leq}N$ for all $j{\geq}1$, then for $i{=}n_c{-}1,n_c{-}2$
\begin{multline*}
\left(\frac{1}{N\Phi(N)^{r-1}}\right)^2\int_0^{\Psi(N) t}\left(2X_{{n_c}-1}^N(u)\pm 1\right)^2\frac{X_1^N(u)X_{i}^N(u)}{N}\diff u\\\leq
\frac{(2N+1)^2}{N^3\Phi(N)^{2r-2}}\int_0^{\Psi(N)t}X_{1}^N(u)X_{i}^N(u)\diff u\leq \frac{9}{N\Phi(N)^{r}}\int_0^{\Psi(N)t}X_{1}^N(u)X_{i}^N(u)\diff u,
\end{multline*}
as a process, the last term converges to $0$ in distribution by Relation~\eqref{relation2} of the induction. 

Similarly, the term
\begin{multline*}
  \left(\frac{1}{N\Phi(N)^{r-1}}\right)^2\int_0^{\Psi(N) t}\Phi(N)X_{{n_c}-1}^N(u)\left(2X_{{n_c}-1}^N(u)-1\right)^2\diff u\\
       {\leq}\frac{4}{N\Phi(N)^{2r-3}}\int_0^{\Psi(N)t}X_{{n_c}-1}^N(u)^2\diff u{\leq}\frac{4}{N\Phi(N)^{r-1}}\int_0^{\Psi(N)t}X_{{n_c}-1}^N(u)^2\diff u
\end{multline*}
converges to $0$ in distribution as a process by Relation~\eqref{relation1} of the induction. Recall that $r{=}\ell$. Therefore, one gets the previsible increasing process, see Relation~\eqref{An0M}, of the martingale
\[
\left(\frac{1}{N\Phi(N)^{r-1}}{M}_{n_c-1,n_c-1}^N(\Psi(N)t)\right)
\]
is converging to $0$ as $N$ goes to infinity. By  Doob's inequality, we obtain that this  martingale is thus vanishing for $N$ large.

Similarly, for $h{\in}\{n_c{-}2,n_c{-}1\}$, the induction assumption gives the convergence in distribution
\[
\lim_{N\to+\infty} \left(\frac{1}{N\Phi(N)^{r-1}}\int_0^{\Psi(N)t}\left(2X_{{n_c}-1}^N(u)\pm 1\right)\frac{X_1^N(u)X_{h}^N(u)}{N}\diff u\right)=(0).
\]
From Proposition~\ref{lem1} we get that, for the convergence in distribution,
\[
\lim_{N\to+\infty} \left(\frac{1}{N\Phi(N)^{r-1}}X_{{n_c}-1}^N(\Psi(N)t)^2 \right)=(0).
\]
By gathering these results in Equation~\eqref{An0} of Section~\ref{SId}, one finally gets that 
\begin{multline*}
\lim_{N\to \infty}\left(\frac{1}{N\Phi(N)^{r-2}}\int_0^{\Psi(N)t}\left(2X_{{n_c}-1}^N(u)-1\right)X_{{n_c}-1}^N(u)\diff u\right){=}(0).
\end{multline*}
For $x{\in}\N$ the relation $x^2{\leq} (2x{-}1)x$ gives therefore that 
\begin{equation}\label{ind2}
  \lim_{N\to \infty}\left(\frac{1}{N\Phi(N)^{r-2}}\int_0^{\Psi(N)t}X_{{n_c}-1}^N(u)X_{h}^N(u)\diff u\right){=}(0),
\end{equation}
for $h{=}{n_c}{-}1$.

Take $k{=}1$ and $h{=}1$. By using Relation~\eqref{SDE1}, one gets
 By Relation~\eqref{relation2} of the induction for $r{+}1$, for any $j{=}1,\dots,n_c{-}1$, $i{=}n_c{-}1,n_c{-}2$ the processes
\[
\left(\frac{1}{N\Phi(N)^r}\int_{0}^{\Psi(N)t}\frac{X_{j}^N(u)}{N}X_{1}^N(u)X_{i}^N(u)\right)\le\left( \frac{1}{N\Phi(N)^r}\int_{0}^{\Psi(N)t}X_{1}^N(u)X_{i}^N(u)\right)
\]
converge in distribution to $(0)$ when $N$ gets large.
By Relation~\eqref{relation1} of the induction for $r{=}\ell{+}1$, for any $j{=}2,\dots,n_c{-}1$, the processes
\[
\left(\frac{\Phi(N)}{N\Phi(N)^r}\int_{0}^{\Psi(N)t}X_{j}^N(u)X_{n_c-1}^N(u)\right)
\]
converge in distribution to $(0)$ when $N$ is converging to infinity. By using similar approach as in the previous case, by replacing $t$ by $\Psi(N)t$ and by multiplying Relation~\eqref{eqx1} of Section~\ref{SId} by $1/(N\Phi(N)^r)$, we obtain the convergence~\eqref{relation2} for $r{=}\ell$ when $k{=}1$.

Now, we prove Relation~\eqref{relation1}  by induction on $h$, from $n_c{-}1$ to $2$. Assume it holds for all $h{\in}\{h'{+}1,\ldots,n_c{-}1\}$.  If Identity~\eqref{eqx4} of Section~\ref{SId} is multiplied by ${1}/({N\Phi(N)^{r-1}})$ and if $t$ is replaced by $\Psi(N)t$, then we show that several of its terms vanish in the limit. They are examined one by one.  \renewcommand{\labelenumi}{\alph{enumi})}
\begin{enumerate}
\item By Proposition~\ref{lem1} and the fact that $r{\geq}2$, one has
  \[
\lim_{N\to+\infty} \left(\frac{1}{N\Phi(N)^{r-1}}X_{{n_c}-1}^N{\cdot}X_{h}^N(\Psi(N)t)\right)=(0).
  \]
\item For $i{\in}\{n_c{-}1, n_c{-}2\}$, 
\[
\frac{1}{N\Phi(N)^{r-1}} \int_0^{\Psi(N)t}X_{h}^N(u) \frac{X_{i}^N(u)X_{1}^N(u)}{N}\,\diff u\leq
\frac{1}{N\Phi(N)^{r-1}} \int_0^{\Psi(N)t}X_{i}^N(u)X_{h}^N(u)  \,\diff u.
\]
Relation~\eqref{relation1} of the induction shows that the last term of this inequality converges in distribution to $0$ when $N$ gets large.
\item For $h{-}1{\ge}1$ ,
\begin{multline*}
\frac{1}{N\Phi(N)^{r-1}} \int_0^{\Psi(N)t}X_{h-1}^N(u) \frac{X_{n_c{-}1}^N(u)X_{1}^N(u)}{N}\,\diff u\\\leq
\frac{1}{N\Phi(N)^{r-1}} \int_0^{\Psi(N)t}X_{n_c{-}1}^N(u) X_{1}^N(u) \,\diff u,
\end{multline*}
the recurrence relation~\eqref{relation2} for $r$ gives that the last term of this relation vanishes as $N$ goes to infinity.
\item For all $n_c{-}1{\geq}i{\geq}h{+}1$, the recurrence assumption gives
  \[
\lim_{N\to+\infty}\left(\frac{1}{N\Phi(N)^{r-2}} \int_0^{\Psi(N)t}X_{i}^N(u) X_{n_c{-}1}^N(u)\,\diff u\right)=(0).
\]
\item The martingale.  The relation
  \[
\croc{\frac{1}{N\Phi(N)^{r-1}}M_{n_c-1,h}^N}(\Psi(N) t)=\frac{1}{N^2\Phi(N)^{2r-2}}\croc{M_{n_c-1,h}^N}(\Psi(N)t)
\]
and, in the same way as before, by checking each term of  the expression~\eqref{crocM} of $\langle M_{n_c-1,h}^N\rangle(t)$, one also gets the convergence in distribution
\[
\lim_{N\to+\infty} \left(\frac{1}{N\Phi(N)^{r-1}}M_{n_c-1,h}^N\left(\Psi(N)t\right)\right)=(0).
\]
\end{enumerate}
One gets finally that the remaining term of Relation~\eqref{eqx4}  of Section~\ref{SId} is also vanishing,  the convergence in distribution
\[
\lim_{N\to \infty}\left(\frac{1}{N\Phi(N)^{r-2}}\int_0^{\Psi(N)t}X_{{n_c}-1}^N(u)X_{h}^N(u)\diff u\right){=}(0)
\]
holds, i.e.\ Relation~\eqref{relation1}  is true for $h$. This gives the proof of this recurrence scheme for $k{=}1$ and all $n_c{-}1{\geq}h{\geq}2$.

To proceed the induction on $k$, from $1$ to $r{-}1$, one uses Equations~\eqref{Anm} and~\eqref{crocM2}, \eqref{eqx1} and~\eqref{eqx5}, and also~\eqref{eqx6} and~\eqref{crocM}  in Section~\ref{SId} below for the processes
\[
\left(\frac{X_{{n_c}-k}^N(\Psi(N)t)^2}{N\Phi(N)^{r-1}}\right),\  \left(\frac{X_1^N{\cdot}X_{n_c-k}^N(\Psi(N)t)}{N\Phi(N)^r}\right),\  \left(\frac{X_{{n_c}-k}^N{\cdot}X_{h}^N(\Psi(N)t)}{N\Phi(N)^{r-1}}\right),
\]
and, with the same method which has been used for the first step one gets that,  for $1{\le} k{\le}(r{-}1)\land(n_c{-}2)$ and  $2{\le} h{\le} n_c{-}1$.
The proposition is proved. 

\subsubsection{Some Stochastic Integral Equations}\label{SId}
For the sake of completeness, we detail the various equations used in the previous proof.  They are obtained by using repeatedly SDE~\eqref{SDE1}, via  stochastic calculus with marked Poisson processes,  see Section~\ref{mpp}. See the proof of Proposition~\ref{lembal} for an example of such a derivation. 

\noindent
For $t{\ge}0$,
\vspace{-3mm}
\begin{multline}\label{An0}
X_{{n_c}-1}^N(t)^2=X_{{n_c}-1}^N(0)^2{+}{M}^N_{{n_c-1,n_c-1}}(t)\\+\lambda_{{n_c}-2}\int_0^{t}\left(2X_{{n_c}-1}^N(u){+}1\right)\frac{X_1^N(u)X_{{n_c}-2}^N(u)}{N}\diff u\\
+\int_0^{t}\left({-}2X_{{n_c}-1}^N(u){+}1\right)\left(\mu_{{n_c}-1}\Phi(N)X_{{n_c}-1}^N(u)+\lambda_{{n_c}-1}\frac{X_1^N(u)X_{{n_c}-1}^N(u)}{N}\right)\diff u
\end{multline}
where $({M}^N_{{n_c-1,n_c-1}}(t))$ is a martingale whose previsible increasing process is given by 
\vspace{-3mm}
\begin{multline}\label{An0M}
\croc{{M}^N_{{n_c-1,n_c-1}}}(t) =
\lambda_{{n_c}-2}\int_0^{t}\left(2X_{{n_c}-1}^N(u){+}1\right)^2\frac{X_1^N(u)X_{{n_c}-2}^N(u)}{N}\diff u\\
{+}\int_0^{t}\left(-2X_{{n_c}-1}^N(u){+}1\right)^2\left(\mu_{{n_c}-1}\Phi(N)X_{{n_c}-1}^N(u){+}\lambda_{{n_c}-1}\frac{X_1^N(u)X_{{n_c}-1}^N(u)}{N}\right)\diff u.
\end{multline}

\medskip

\noindent
For  $2{\le}k{\le}n_c{-}2$ and $t{\ge}0$,
\vspace{-3mm}
\begin{multline}\label{Anm}
X_{n_c-k}^N(t)^2{=}X_{n_c-k}^N(0)^2+M^N_{n_c-k,n_c-k}(t)\\+\lambda_{n_c-k{-}1}\int_0^{t}\left(2X_{n_c-k}^N(u){+}1\right)\frac{X_1^N(u)X_{n_c-k-1}^N(u)}{N}\diff u\\
+\sum_{i=n_c-k+1}^{{n_c}-1}\mu_i\Phi(N)\int_0^{t} \int_{y\in{\cal S}_{i}}\left(2X_{n_c-k}^N(u)y_{n_c-k}{+}y_{n_c-k}^2\right)\nu_i(\diff y) X_{i}^N(u)\,\diff u\\
+\int_0^{t}\left(1{-}2X_{n_c-k}^N(u)\right)\left(\mu_{n_c-k}\Phi(N)X_{n_c-k}^N(u){+}\lambda_{n_c-k}\frac{X_1^N(u)X_{n_c-k}^N(u)}{N}\right)\diff u,
\end{multline}
$(M^N_{n_c-k,n_c-k}(t))$ is a martingale whose previsible increasing process is given by
\vspace{-3mm}
\begin{multline}\label{crocM2}
  \croc{M_{n_c-k,n_c-k}^N}(t){=}\lambda_{n_c-k{-}1}\int_0^{t}\left(2X_{n_c-k}^N(u){+}1\right)^2\frac{X_1^N(u)X_{n_c-k-1}^N(u)}{N}\diff u\\
+\sum_{i=n_c-k+1}^{{n_c}-1}\mu_i\Phi(N)\int_0^{t} \int_{y\in{\cal S}_{i}}\left(2X_{n_c-k}^N(u)y_{n_c-k}{+}y_{n_c-k}^2\right)^2\nu_i(\diff y) X_{i}^N(u)\,\diff u\\
+\int_0^{t}\left(1{-}2X_{n_c-k}^N(u)\right)^2\left(\mu_{n_c-k}\Phi(N)X_{n_c-k}^N(u){+}\lambda_{n_c-k}\frac{X_1^N(u)X_{n_c-k}^N(u)}{N}\right)\diff u. 
\end{multline}

\medskip

\noindent
For $t{\ge}0$,
\vspace{-3mm}
\begin{multline}\label{eqx1}
X_{{n_c}-1}^N(t)X_{1}^N(t)=X_{{n_c}-1}^N(0)X_{1}^N(0)+M_{n_c-1,1}^N(t)\\
+\int_0^{t}X_{1}^N(u)\left(\lambda_{{n_c}-2}\frac{X_{{n_c}-2}^N(u)X_{1}^N(u)}{N}-\lambda_{{n_c}-1}\frac{X_{{n_c}-1}^N(u)X_{1}^N(u)}{N}\right)\diff u\\
+\mu_{n_c-1}\Phi(N)\int_0^{t}\left[\rule{0mm}{4mm}{-}X_{1}^N(u){+}\croc{\nu_{n_c-1},I_{1}}\left(X_{{n_c}-1}^N(u)-1\right)\right]X_{n_c-1}^N(u)\diff u\\
+\int_0^{t}X_{{n_c}-1}^N(u)\Bigg(-\sum_{j=1}^{n_c-1}(1+\ind{j=1})\lambda_{j}\frac{X_{j}^N(u)X_{1}^N(u)}{N}
+\sum_{i=2}^{n_c-2}\mu_i\Phi(N)X_i^N(u)\croc{\nu_{i},I_{1}}\Bigg)\diff u\\
+\int_0^{t}\left(-\lambda_{n_c-2}\frac{X_{n_c-2}^N(u)X_{1}^N(u)}{N}+\lambda_{n_c-1}\frac{X_{n_c-1}^N(u)X_{1}^N(u)}{N}\right)\diff u,
\end{multline}
where $(M_{n_c-1,1}^N(t))$ is a martingale.

\medskip

\noindent
For $ 2{\le}h{\leq}n_c{-}2$ and $t{\ge}0$,
\vspace{-2mm}
\begin{multline}\label{eqx4}
X_{{n_c}-1}^N(t)X_{h}^N(t)=X_{{n_c}-1}^N(0)X_{h}^N(0)+M_{n_c-1,h}^N(t)\\
+\int_0^{t}X_{h}^N(u)\left(\lambda_{{n_c}-2}\frac{X_{{n_c}-2}^N(u)X_{1}^N(u)}{N}-\lambda_{{n_c}-1}\frac{X_{{n_c}-1}^N(u)X_{1}^N(u)}{N}\right)\diff u\\
+\mu_{n_c-1}\Phi(N)\int_0^{t}\left[\rule{0mm}{4mm}{-}X_{h}^N(u){+}\croc{\nu_{n_c-1},I_{h}}\left(X_{{n_c}-1}^N(u)-1\right)\right]X_{n_c-1}^N(u)\diff u\\
+\int_0^{t}X_{{n_c}-1}^N(u)\Bigg(\lambda_{h-1}\frac{X_{h-1}^N(u)X_{1}^N(u)}{N}-\lambda_{h}\frac{X_{h}^N(u)X_{1}^N(u)}{N}\\
\hspace{5cm}-\mu_{h}\Phi(N)X_{h}^N(u)+\sum_{i=h+1}^{n_c-2}\mu_i\Phi(N)X_i^N(u)\croc{\nu_{i},I_{h}}\Bigg)\diff u\\
-\ind{h=n_c-2}\lambda_{{n_c}-2}\int_0^{t}\frac{X_{{n_c}-2}^N(u)X_1^N(u)}{N}\diff u,
\end{multline}
where $(M_{n_c-1,h}^N(t))$ is a martingale.

\medskip

\noindent
Additional identities used to complete the proof of the proposition. \\
For $1{<}h{\leq}n_c{-}1$ and $t{\ge}0$,
\vspace{-5mm}
\begin{multline}\label{eqx5}
X_1^N(t)X_{h}^N(t){=}X_1^N(0)X_{h}^N(0)+ M_{1,h}^N(t)\\
{+}\lambda_{h-1}\int_0^{t}\left(X_1^N(u){-}1{-}\ind{h{=}2}\right)\frac{X_{h-1}^N(u)X_1^N(u)}{N}  \diff u\\
{+}\lambda_{h}\int_0^{t}\left(1{-}X_1^N(u)\right)\frac{X_{h}^N(u)X_1^N(u)}{N}   \diff u\\
-\sum_{j=2}^{n_c-1} \lambda_{j}\int_0^{t}X_h^N(u)\frac{X_{j}^N(u)X_1^N(u)}{N}  \diff u
-2\lambda_1\int_0^{t} X_h^N(u)\frac{X_1^N(u)^2}{N}\ind{X_1^N(u)\ge 2} \diff u\\
+\sum_{i=2}^{{n_c}-1}\mu_i\Phi(N)\int_0^{t}\int_{\cal{S}_i}\left[y_1X_h^N(u){+}y_h\left(\rule{0mm}{4mm} X_1^N(u){+}y_1\right)\ind{i{>}h}{-}\left(X_1^N(u){+}y_1\right)\ind{i=h}\right]\\\nu_i(\diff y)  X_i^N(u)\diff u,
\end{multline}
where $(M_{1,h}^N(t))$ is a martingale whose previsible increasing process is given by
\vspace{-3mm}
\begin{multline}\label{crocM1h}
  \croc{M_{1,h}^N}(t){=}\sum_{j\neq 1, h,h-1}\frac{\lambda_j}{N}\int_0^tX_h^N(u)^2X_j^N(u)X_1^N(u)\diff u\\
  +\lambda_{h}\int_0^{t}\left(1{-}X_1^N(u){-}X_h^N(u)\right)^2\frac{X_{h}^N(u)X_1^N(u)}{N}   \diff u\\
  +\ind{h=2}\left(\lambda_1\int_0^{t} \left(-2X_2^N(u)+X_1^N(u)-2\right)^2\frac{X_1^N(u)^2}{N}\ind{X_1^N(u)\ge 2}\diff u\right)\\
  +\ind{h>2}\left(\lambda_1\int_0^{t}\hspace{-2mm} 4X_h^N(u)^2\frac{X_1^N(u)^2}{N}\ind{X_1^N(u)\ge 2}\!\diff u\right.\\\left. {+}\lambda_{h-1}\int_0^{t}\hspace{-2mm}\left(X_1^N(u){-}X_h^N(u){-}1\right)^2\frac{X_{h-1}^N(u)X_1^N(u)}{N}\!\diff u\right)\\
  +\sum_{i=2}^{{n_c}-1}\mu_i\Phi(N)\int_{[0,t]{\times}\cal{S}_i}\left[y_1X_h^N(u){+}y_h\left(X_1^N(u){+}y_1\right)\ind{i{>}h}\right.\\\left.{-}\left(X_1^N(u){+}y_1\right)\ind{i=h}\right]^2\nu_i(\diff y)  X_i^N(u)\diff u,
\end{multline}
\medskip

\noindent
For $2{\le}k{\le}n_c{-}2$, $2{\le}h{<}n_{c}{-}k$ and $t{\ge}0$,
\begin{multline}\label{eqx6}
X_{{n_c}-k}^N(t)X_{h}^N(t)=X_{{n_c}-k}^N(0)X_{h}^N(0)+M_{n_c-k,h}^N(t)\\
+\int_0^{t}X_{h}^N(u)\left(\lambda_{{n_c}-k-1}\frac{X_{{n_c}-k-1}^N(u)X_{1}^N(u)}{N}
-\lambda_{{n_c}-k}\frac{X_{{n_c}-k}^N(u)X_{1}^N(u)}{N}\right)\diff u\\
+\mu_{n_c-k}\Phi(N)\int_0^{t}\left(\rule{0mm}{4mm}{-}X_{h}^N(u)
{+}\croc{\nu_{n_c-k},I_{h}}\left(X_{{n_c}-k}^N(u)-1\right)\right)X_{n_c-k}^N(u)\diff u\\
+\int_0^{t}X_{{n_c}-k}^N(u)\Bigg(\lambda_{h-1}\frac{X_{h-1}^N(u)X_{1}^N(u)}{N}
-\lambda_{h}\frac{X_{h}^N(u)X_{1}^N(u)}{N}\\
-\mu_{h}\Phi(N)X_{h}^N(u)
+\sum_{i=h+1}^{n_c-k-1}\mu_i\Phi(N)X_i^N(u)\croc{\nu_{i},I_{h}}\Bigg)\diff u\\
{+}\int_0^t\hspace{-2mm}\sum_{i=n_c-k+1}^{n_c-1}\hspace{-2mm}\mu_i\Phi(N)X_i^N(u)\left(\croc{\nu_{i},I_{h}}X_{n_c-k}^N(u){+}\croc{\nu_{i},I_{n_c-k}}X_h^N(u){+}\croc{\nu_{i},I_{n_c-k}}\croc{\nu_{i},I_{h}}\right)\diff u\\
-\ind{h=n_c-k-1}\lambda_{{n_c}-k-1}\int_0^{t}\frac{X_{{n_c}-k-1}^N(u)X_1^N(u)}{N}\diff u,
\end{multline}
where $(M_{n_c-k,h}^N(t))$ is a martingale whose previsible increasing process is given by
\vspace{-3mm}
\begin{multline}\label{crocM}
\croc{M_{n_c-k,h}^N}(t)\\{=}\ind{h{=}n_c{-}k{-}1}\lambda_{{n_c}-k-1}\int_0^{t}\left(X_{{n_c}-k-1}^N(u){-}X_{{n_c}-k}^N(u)-1\right)^2\frac{X_{{n_c}-k-1}^N(u)X_1^N(u)}{N}\diff u\\
+\int_0^{t}X_{h}^N(u)^2\left(\lambda_{{n_c}-k-1}\frac{X_{{n_c}-k-1}^N(u)X_{1}^N(u)}{N}\ind{h<n_c-k-1}{+}\lambda_{{n_c}-k}\frac{X_{{n_c}-k}^N(u)X_{1}^N(u)}{N}\right)\diff u\\
+\mu_{n_c{-}k}\Phi(N)\int_0^{t}\int_{y\in{\cal S}_{n_c-k}} \left(X_{h}^N(u){-}y_h (X_{{n_c}-k}^N(u)-1)\right)^2X_{n_c-k}^N(u)\,\nu_{n_c-k}(\diff y)\diff u\\
+\int_0^{t}X_{{n_c}-k}^N(u)^2\Bigg(\lambda_{h-1}\frac{X_{h-1}^N(u)X_{1}^N(u)}{N}{+}\lambda_{h}\frac{X_{h}^N(u)X_{1}^N(u)}{N}\ind{h<n_c-k-1}\\
{+}\mu_{h}\Phi(N)X_{h}^N(u)+
\sum_{i=h+1}^{n_c-k-1}\mu_i\Phi(N)X_i^N(u)\croc{\nu_{i},I_{h}^2}
\Bigg)\,\diff u\\
+\sum_{n_c-k+1}^{n_c-1} \mu_i \Phi(N)\int_{{\cal S}_{i}} \int_0^{t}\left(y_hX_{{n_c}-k}^N(u)+y_{n_c-k}X_{h}^N(u)+y_hy_{n_c-k}\right)^2 X_i^N(u)\nu_{i}(\diff y) \,\diff u.
\end{multline}

\end{document}